\documentclass[11pt]{amsart}

\usepackage[backref=page]{hyperref}

\usepackage[latin9]{inputenc}
\usepackage{color}

\usepackage{amsfonts}
\usepackage{amssymb}
\usepackage{amsmath}
\usepackage{amsthm}
\usepackage{geometry}
\usepackage{verbatim}
\usepackage[toc,page]{appendix}
\usepackage{todonotes}
\usepackage[all]{xy}
\usepackage{bbm}
\usepackage[capitalise]{cleveref}

\usepackage[OT2,T1]{fontenc}
\DeclareSymbolFont{cyrletters}{OT2}{wncyr}{m}{n}
\DeclareMathSymbol{\Sha}{\mathalpha}{cyrletters}{"58}

\newcommand{\mG}{\mathbb{G}}
\newcommand{\Z}{\mathbb{Z}}
\newcommand{\eps}{\varepsilon}

\DeclareMathOperator\Gal{Gal}
\DeclareMathOperator\Br{Br}
\DeclareMathOperator\Aut{Aut}

\DeclareMathOperator\GL{GL}
\DeclareMathOperator\SL{SL}

\DeclareMathOperator\Image{Im}
\DeclareMathOperator\HG{H}
\DeclareMathOperator\Pic{Pic}

\DeclareMathOperator\Zd{Z}
\DeclareMathOperator\HGH{\hat H}
\DeclareMathOperator\Hom{Hom}
\DeclareMathOperator\Ext{Ext}

\DeclareMathOperator\BC{BC}

\DeclareMathOperator\ZG{Z}

\DeclareMathOperator\res{res}

\DeclareMathOperator\Res{R}

\DeclareMathOperator\coker{coker}
\DeclareMathOperator\Ind{Ind}
\DeclareMathOperator\Inv{inv}
\DeclareMathOperator\spec{Spec}
\DeclareMathOperator\charak{char}
\DeclareMathOperator\Dec{Dec}

\DeclareMathOperator\Map{Map}

\newcommand{\sep}[1]{{#1}^{\text{sep}}}
\newcommand{\sub}[1]{\mbox{}_{#1}}
\newcommand{\Spec}{\spec}
\newcommand{\til}{\widetilde}

\begin{document}
\providecommand{\keywords}[1]{\textbf{\textit{Keywords: }} #1}
\newtheorem{theorem}{Theorem}[section]
\newtheorem{lemma}[theorem]{Lemma}
\newtheorem{observ}[theorem]{Observation}
\newtheorem{prop}[theorem]{Proposition}
\newtheorem{cor}[theorem]{Corollary}
\newtheorem{problem}[theorem]{Problem}
\newtheorem{question}[theorem]{Question}
\newtheorem{conjecture}[theorem]{Conjecture}
\newtheorem{claim}[theorem]{Claim}
\newtheorem{defn}[theorem]{Definition} 
\theoremstyle{remark}
\newtheorem{remark}[theorem]{Remark}
\newtheorem{example}[theorem]{Example}
\newtheorem{condenum}{Condition}

\newcommand{\cc}{{\mathbb{C}}}   
\newcommand{\mQ}{{\mathbb{Q}}}   
\newcommand{\mP}{{\mathbb{P}}}   
\newcommand{\ff}{{\mathbb{F}}}  
\newcommand{\nn}{{\mathbb{N}}}   
\newcommand{\qq}{{\mathbb{Q}}}  
\newcommand{\rr}{{\mathbb{R}}}   
\newcommand{\zz}{{\mathbb{Z}}}  
\newcommand{\fp}{{\mathfrak{p}}}
\newcommand{\fP}{{\mathfrak{P}}}
\newcommand{\fq}{{\mathfrak{q}}}
\newcommand{\fr}{{\mathfrak{r}}}
\newcommand{\mS}{{\mathcal{S}}}
\newcommand{\ra}{{\rightarrow}}
\newcommand{\divides}{\,|\,}

\newcommand{\X}{X}

\newcommand\DN[1]{{\color{blue} {#1}}}
\newcommand\DK[1]{{\color{red} {#1}}}
\newcommand\oline[1] {{\overline{#1}}}

\def\technion{Department of Mathematics, Technion - Israel Institute of Technology, Haifa, Israel}
\def\penn{Department of Mathematics, University of Pennsylvania, PA, USA}

\title[Rational connectedness and parametrization]{On rational connectedness and parametrization of finite Galois extensions}
\author{Danny Krashen}
\address{\penn}
\email{dkrashen@upenn.edu}
\author{Danny Neftin}
\address{\technion}
\email{dneftin@technion.ac.il}%

\begin{abstract}
Given two $G$-Galois extensions of $\mQ$, is there an extension of $\mQ(t)$ that specializes to both? The equivalence relation on $G$-Galois extension of $\mQ$,  induced by the above question, is called $R$-equivalence. The number of $R$-equivlance classes indicates how many rational spaces are required in order to parametrize all $G$-Galois extensions of $\mQ$. We determine the $R$-equivalence classes  for  basic families of groups $G$,
and consequently obtain parametrizations of the $G$-Galois extensions of $\mQ$ in the absence of a generic extension for $G$. 
\end{abstract}
\maketitle

\section{Introduction}
The geometry of parametrizing spaces of $G$-Galois extensions  over a given number field $K$ is a subject of major interest in arithmetic geometry \cite{Sal, BR,  Deb, Wit}. For a fixed finite group $G$, such parametrizing spaces $X$ always exist, and are represented by versal $G$-torsors $Y\ra X$ \cite{BR}, but their geometry and their sets of rational points $X(K)$ are far from understood. Properties of interest include the rationality of $X$, cf.\ Noether's problem \cite{Wit}; and the weaker property of retract rationality of $X$ \cite{Sal}, or equivalently the existence of a generic extension for $G$ over $K$ \cite{Sal} or \cite[Chp.\ 5]{JLY}.

In this paper, we consider {\it $R$-equivalence} (a.k.a.\ rational connectedness) on such parametrizing spaces $X$  for $G$-Galois extensions of $K$. 
More precisely,  we represent $G$-Galois extensions  as morphisms in $\HG^1(K,G)=\Hom(\Gamma_K,G)/\sim$, where $\Gamma_K$ is the absolute Galois group of $K$, and $\sim$ is equivalence by inner automorphisms of $G$. We then say $\alpha,\beta\in \HG^1(K,G)$ are {\it $R$-equivalent} over $K$ if there exist $\gamma_1(t),\ldots,\gamma_r(t)\in\HG^1(K(t),G)$, $r\geq 1$ with specializations:  
 $\gamma_1(0)=\alpha$, $\gamma_r(1)=\beta$,  and $\gamma_i(1)=\gamma_{i+1}(0)$ for $i=1,\ldots,r$, cf.\ \S \ref{sec:setup}.
Note that a single  $\gamma_i(t)$ corresponds to a $G$-Galois extension $E_i/K(t)$ whose specializations at $t=0$ and $1$ are the $G$-Galois extensions corresponding to $\gamma_i(0)$ and $\gamma_i(1)$, resp., so that $E_1,\ldots,E_r/K(t)$ is a sequence of $G$-Galois extensions connecting the $G$-Galois extensions corresponding to $\alpha$ and $\beta$.  This notion of $R$-equivalence coincides with the classical notion of $R$-equivalence  for versal torsors for $G$, see Lemma \ref{lem:BG-T}. The classical notion of $R$-equivalence has been studied on tori \cite{CTS2}, algebraic groups \cite{Gille}, and when $G$ is finite over function fields $K$ over $p$-adic fields \cite{MB}, but little is known concerning these parameterizing spaces/versal torsors for $G$-Galois extensions of number fields.
 
$R$-equivalence using a single extension $E_1/K(t)$, that is, with $r=1$,  amounts to the following property for pairs of extensions: \\
 {\bf $G$-BB$^{n}_K$}:  every $n$ classes in $\HG^1(K,G)$ are the specializations of some $\gamma(t)\in \HG^1(K(t),G)$.\\
This property is analogous to the Beckmann-Black lifting property for a single extension \cite{BEC}, or Property $G$-\textbf{BB} in \cite{Deb2}. The property {\bf $G$-BB$^2_K$} was first raised by Harbater, and examples of fields $K$ and groups $G$ for which it does not occur were given in \cite[Appendix A]{CT2}, namely, with $K$ a $2$-adic field and $G=C_8$ the cyclic group of order $8$, or with $K=F((t))$ for the field  $F$ of invariants of a certain $p$-group $G$. However,  little is known about $R$-equivalence and {\bf $G$-BB$^n_K$} over number fields $K$ in the absence of a generic polynomial for $G$: 

For abelian groups $G$ and $K=\mQ$, 
the existence of a generic (resp.\ parametric) extension for odd  (resp.\ $2$-power) order $G$ over $\mQ$  \cite[\S 5.3]{JLY} (resp.\ \cite{Schneps}) implies  that $\HG^1(\mQ,G)$ has {\bf $G$-BB$^2_K$} and in particular is {\it $R$-trivial}, that is, has one $R$-equivalence class. However, already for quadratic fields $K$ and the cyclic group $G=C_{2^s}$, the number of $R$-equivalence classes has so far been unknown. For Dihedral groups $G=D_{2^s}$ of order $2^{s+1}$,  it is known that $\HG^1(K,D_{2^s})$, $s=2,3$, have {\bf $G$-BB$^2_K$} since a generic extension exists \cite{Black2,Led2, CSK}. However, so far the $R$-triviality of $\HG^1(K,D_{2^s})$ has been unknown for all $s\geq 4$. 


In the paper we study the number $r=r(K,G):=\HG^1(K,G)/R$ of $R$-equivalence class on $\HG^1(K,G)$ for basic families of groups $G$ when a generic extension does not exist or is unknown to exist. We show that  this $r$ is sometimes also the minimal number of extensions of a rational function field $F=K(t_1,\ldots,t_d)$, $d>0$ required to parametrize all $G$-Galois extensions, allowing the  parametrization of $G$-Galois extensions for groups $G$ that lack a generic extension. In similarity to \cite{DKLN1}, say  $\{\gamma_1,\ldots,\gamma_r\}\subseteq \HG^1(F,G)$ is a {\it densely parametric set} for $\HG^1(K,G)$  if every  $\alpha\in \HG^1(K,G)$ is a specialization of some $\gamma_i$ in every open subset of a model $X$ for $F$, cf.\ Definition \ref{def:param}. 

Already when $G$ is abelian,   $r>1$ is possible. All such  nontrivial examples come from cyclic $2$-groups $G=C_{2^s}$ where $r$ is given as follows.  
Let $\mu_n=\langle\zeta_n\rangle$ denote the $n$-th roots of unity, $\charak K$ the characteristic of $K$, and $\Br(L/K)$  the relative Brauer group of $L/K$. 
\begin{theorem}\label{thm:2-power-cyclic}
Let $K$ be a field of $\charak{K} \neq 2$, let $s\geq 3$ be an integer, and $E\subseteq K(\mu_{2^s})$  the subfield fixed by conjugation $\zeta_n\to\zeta_n^{-1}$. 
Then  the number of $R$-equivalence classes $r=\#\HG^1(K,C_{2^s})/R$  is the cardinality of the quotient of $\Br(K(\mu_{2^s})/K)$ by $\Br(E/K)+\Br(K(\sqrt{-1})/K)$. 
Furthermore, this $r$ is the minimal cardinality of a parametric set   for $\HG^1(K,C_{2^s})$ over rational function fields $K(t_1,\ldots,t_d)$, $d\in \mathbb N$. 
\end{theorem}
Note that explicit $r(K,C_{2^s})$ extensions parametrizing $\HG^1(K,C_{2^s})$ were given in \cite[Theorem 4.2]{Schneps}.  Theorem \ref{thm:2-power-cyclic} shows this is the minimal number of such extensions. Note that for number fields $K$, the subgroup $\Br(E/K)+\Br(K(\sqrt{-1})/K)$ coincides with Tignol's decompsable subgroup $\Dec(K(\mu_{2^s})/K)$, see \cite{Tig}.

In pparticular, for $G=C_{2^s}$, the theorem shows that $\HG^1(K,G)$ has the following Beckmann--Black-type property: Every $\alpha_1,\ldots,\alpha_n\in \HG^1(K,G)$, $n\geq 2$ which are $R$-equivalent are the specialization of some $\alpha\in \HG^1(K(t),G)$.  
For an abelian group $A=O\times \prod_{i=1}^k C_{2^{s_i}}$, where $O$ is of odd order and $s_i\geq 1$ are integers, it follows that the number of $R$-equivalence classes is $\prod_{i=1}^kr(K,C_{2^{s_i}})$. Already for $A=C_8$ and $K=\mQ(\sqrt{-17})$, the theorem yields examples of non-$R$-equivalent  classes in  $\HG^1(K,A)$. In particular, there exist $C_8$-Galois extensions $L_1,L_2/K$   for which there exists no $C_8$-Galois extension $L/K(t)$ that specializes to both  $L_1$ and $L_2$, see \S\ref{sec:exam1}. 
As already noted in \cite[Ex.\ 5.5]{DKLN2},  this implies there exists no $G$-Galois extension $L/K(t_1,\ldots,t_d)$ that specializes to all $G$-Galois extensions of $K$ for any $d>0$.


The second family of groups we consider is the family of semidirect products $G=C\rtimes H$ whose kernel $C$ is cyclic of odd order. Here, as well, the number of $R$-equivalence classes gives the cardinality of a minimal parametrizing set: 
\begin{theorem}\label{thm:semisimple} Let $K$ be a field. Suppose that a finite group $H$ acts on a cyclic group $C$ of odd order coprime to $\charak K$.  Then the number $\#\HG^1(K,C\rtimes H)/R$ of $R$-equivalence classes  coincides with that on $\HG^1(K,H)$. If furthermore $\HG^1(K,H)$ has a parametrizing set of cardinality $r$ over a $K$-rational function field, then so does $\HG^1(K,C\rtimes H)$. 
\end{theorem} 
As $H$ and hence $G=C\rtimes H$ may not have a generic extension, this provides further examples of finite parametrizing sets in the absence of a generic extension.  
Note that $C$ may be replaced by an odd order abelian group as long as the action of $H$ is {\it semisimple}, that is, $C$ decomposes as a direct sum $\bigoplus_{i=1}^k C_i$ of $H$-invariant cyclic subgroups $C_i$, $1,\ldots,k$. Furthermore, the group $C$ in Theorem \ref{thm:semisimple} can be replaced by an iterated semidirect product $(C_k\rtimes (\cdots\rtimes (C_1\rtimes H)))$ with odd-order cyclic kernels. 

In the opposite scenario where $A$ is a cyclic $2$-group, 
the analysis of $R$-equivalence on $\HG^1(K,C\rtimes H)$ is much more difficult, and even its finiteness is unknown. 
Among these, we focus on the basic family of Dihedral groups $D_{2^s}=C_{2^s}\rtimes C_2$, $s\geq 4$ where a generic polynomial is unknown to exist and even the Beckmann-Black problem is open. 
\begin{theorem}\label{thm:main}
Let $s\geq 2$ and $K$ be a number field. 
Then $\HG^1(K,D_{2^s})$ is $R$-trivial. 
\end{theorem}
Although the proof is given only for dihedral groups, we expect our approach, described below, to carry over to certain other actions, 
cf.\ Remark \ref{rem:noncyclic}, yielding: 
\begin{problem}
    Let $G$ be a finite solvable group and $K$ a number field. Is $\HG^1(K,G)/R$ finite?
    In particular, is finiteness preserved under split extensions with abelian kernel?
\end{problem}

In contrast to the $R$-triviality in Theorem \ref{thm:main}, when $C$ is allowed to be noncyclic abelian,  we give examples where  $\#\HG^1(K,H)/R$ is $R$-trivial but  $\HG^1(K,C\rtimes H)/R$ contains new nontrivial classes, see \S\ref{sec:example2}. This suggests that a bound on $\HG^1(K,G)/R$ would depend on the Brauer-Manin obstruction and weak approximation defect, cf.\ \cite[\S 7]{CTS2}. 

Our approach considers $R$-equivalence on $A$-torsors in $\HG^1(K,A)$, for finite abelian $K$-groups $A$, that is, finite abelian  
group schemes over $K$. Such $K$-groups are equipped with a (profinitely) continuous action of $\Gamma_K$,  corresponding to $(A\rtimes H)$-Galois extensions of $K$, where $H\leq \Aut(A)$ is the image of the action of $\Gamma_K$ on $A$, and the corresponding $(A\rtimes H)$-extensions contain a fixed $H$-extension defined by the surjection $\Gamma_K\ra H$, see \S\ref{sec:setup}. 


Our first observation relates $R$-equivalence on $\HG^1(K,A)$ to $R$-equivalence on algebraic tori, a theory developed by Colliot--Th\'el\`ene--Sansuc \cite{CTS2}. We first observe that the finiteness of $R$-equivalence classes on tori over finitely generated fields $K$ \cite[\S5, Cor.\ 2]{CTS2} implies that of $\HG^1(K,A)$, 
and  furthermore provides a parametrizing set 
\footnote{The action of $\Gamma_{K(t_1,\ldots,t_d)}$ on $A$ is taken to factor through that of $\Gamma_K$ to ensure their  compatibility.}:
\begin{observ}
    Let $K$ be a finitely generated field, and $A$ a finite $K$-group of order prime to $\charak K$. Then the number $r$ of $R$-equivalence classes in $\HG^1(K,A)/R$ is finite, and  it is the minimal number of torsors over $K(t_1,\ldots,t_d)$ needed to densely parametrize $\HG^1(K,A)$,  
    for all sufficiently large\footnote{Clearly such torsors may not exist for small values of $d$.} $d\in\mathbb N$.
\end{observ}
The observation follows directly from  Corollary \ref{cor:param}. 
Note that since a single parametrizing extension and hence a generic polynomial cannot exist if $r > 1$, this approach  yields the missing parametrizing extensions when $\HG^1(K,A)$ is not $R$-trivial.

For $A=C_{2^s}$, we compute the corresponding tori explicitly and deduce Theorem \ref{thm:2-power-cyclic} from the above approach. We combine this approach with a theorem of Endo--Miyata concerning the $R$-triviality of tori with Sylow-cyclic splitting fields to deduce that $\HG^1(K,A)$ is $R$-trivial for cyclic $K$-groups $A$ and furthermore that the corresponding field of invariants is stably rational, see Proposition \ref{prop:cyclic-odd}. This is the key to proving Theorem \ref{thm:semisimple}.

To prove Theorem \ref{thm:main}, we consider $2$-power cyclic $K$-groups $A$. 
When the $\Gamma_K$-action factors through $H=C_2=\Gal(K[\sqrt{a}]/K)$ and acts on $C$ by inversion,  a class $\alpha\in\HG^1(K,A)$ corresponds to a $D_{2^s}$-Galois extension containing $K[\sqrt{a}]$.  
We then determine 
$\HG^1(K,A)/R$ similarly to Theorem \ref{thm:2-power-cyclic}: 
\begin{theorem}\label{thm:2-power}
Let $K$ be a field of $\charak K \neq 2$, let $s\geq 3$ be an integer, $a\in K^\times$  nonsquare, and $E\subseteq K(\mu_{2^s})$ the subfield fixed by conjugation. 
Let $A^{(a)}=C^{(a)}_{2^s}$, $s\geq 3$ be  the cyclic  $K$-group of order $2^s$ equipped with the $\Gamma_K$-action  through $\Gal(K[\sqrt{a}]/K)$ by inversion. 
Then  $\HG^1(K,A^{(a)})/R$ is in bijection with the quotient of $\Br(E[\sqrt{-a}]/K)$ by $\Br(E/K)+\Br(K[\sqrt{-a}]/K)$.  
\end{theorem}
In particular, as mentioned above there is a densely parameterizing set of cardinality $r=\#\HG^1(K,A^{(a)})/R$ for $\HG^1(K,A^{(a)})$ over $K(t_1,\ldots,t_d)$, for some $d\in\mathbb N$. 
Over number fields $K$, the quotient subgroup is explicitly as Tignol's decomposable subgroup of $\Br(E[\sqrt{-a}]/K)$, cf.\ Remark \ref{rem:noncyclic}.(3). The number 
$r$ of such $R$-equivalence classes is then  
 $2^{m-1}$, where $m=m(K,s,a):=\#\{\fp\text{ prime of }K\,|\,K_\fp(\eta_s,\sqrt{-a})/K_\fp\text{ is noncyclic}\}$, by Lemma \ref{lem:local} and Remark \ref{rem:p-adic}. A similar description applies to Theorem \ref{thm:2-power-cyclic}.

Much of the difficulty in proving Theorem \ref{thm:main} arises from the following consequence of Theorem \ref{thm:2-power}: 
As one varies through the quadratic extensions 
$K[\sqrt{a}]/K$, the number of $R$-equivalence classes on $\HG^1(K,A^{(a)})$ may be arbitrarily large! 
The proof of Theorem \ref{thm:main} then varies the Galois action on $A$ through specializations $t\mapsto t_0\in K$  of a quadratic extension $K(t,\sqrt{p(t)})/K(t)$, $p(t)\in K[t]$, 
see \S\ref{sec:dihedal}. 

Theorem \ref{thm:2-power} also yields examples of $p$-adic fields with odd $p$ for which $\HG^1(K,A^{(a)})$ is not $R$-trivial,   answering  a question of Colliot--Th\'el\`ene from the end of \cite{CT2}, when the group is allowed to be a nonconstant $K$-group, see Example  \ref{exam:CT}.

\subsubsection*{Acknowledgements} 
The authors were supported by the Israel Science Foundation, grant no.~353/21. 

\section{Preliminaries}
\subsection{Setup}\label{sec:setup}
{\it $K$-groups.} Let $K$ be a field with separable closure $\sep K$ and absolute Galois group $\Gamma_K$. 
We shall consider \`etale algebra extensions of $K$ and their Galois theory following \cite[Chp.\ III]{DI}. 
Throughout the paper, a $G$-Galois extension of $K$ is an \'etale algebra which is Galois over $K$ of group $G$. For a fixed polynomial $f\in K[t]$,  the ring $K[\alpha]$, where $\alpha$ is a root of $f$,  denotes the algebra $K[x]/(f)$, so that  $K[\sqrt{a}]:=K[x]/(x^2-a)$. 

We shall call 
group schemes over $K$ for short $K$-groups. We note that torsors for a finite $K$-group $A$ are themselves finite $K$-groups (as they are \'etale locally isomorphic to $A$), and so can be written in the form $\Spec L$ for $L/K$ a finite \'etale $K$-algebra. More explicitly, via Galois descent, if we are given a (profinitely continuous) cocycle $\alpha: \Gamma_K \to A(\sep K)$ we can describe the corresponding $K$-algebra $L_\alpha$ as the $\Gamma_K$-invariants in the algebra $\prod_{a \in A(\sep K)} \sep K$ via the action $\sigma \cdot_\alpha x \equiv \alpha(\sigma) \sigma(x)$, where $\sigma(x)$ denotes the coordinatewise action of the absolute Galois group and where $\alpha(\sigma) \in A(\sep K)$ acts on a tuple in $\prod_{a \in A(\sep K)} \sep K$ by permuting the entries via the action of $A(\sep K)$ on itself. For a (constant) finite group $G$ with a subgroup $H < G$,  homomorphisms $\alpha\in \HG^1(K,G)=\Hom(\Gamma_K,G)/\sim$ with image $H$, up to equivalence $\sim$ by conjugation in $G$, correspond to \'etale algebras $\bigoplus_{H\backslash G} E $, where $E=\oline K^{\ker\alpha}$ is the fixed field of $\ker\alpha\leq\Gamma_K$ and $\Gamma_K$ acts by  $\sigma\cdot  (x_\tau)_{\tau\in H\backslash G}= (\sigma( x_{\tau\alpha(\sigma))})_{\tau\in H\backslash G}$. 


For a finite abelian $K$-group $A$, the action of $\Gamma_K$ factors through a finite Galois group $H=\Gal(L/K)$. Furthermore, there is a well known correspondence between (profinitely continuous) cocycles $\alpha\in \Zd^1(K,A):=\Zd^1(\Gamma_K,A)$ with solutions to the embedding problem $A\rtimes H\ra \Gal(L/K)$, that is, homomorphisms $\Gamma_K\ra A\rtimes H$ whose composition with the projection $A\rtimes H\ra H=\Gal(L/K)$ coincides with the restriction map $\Gamma_K\ra \Gal(L/K)$ \cite[Prop.\ 3.5.11]{NSW}. Moreover,  classes in $\HG^1(K,A):=\HG^1(\Gamma_K,A)$ are in one to one correspondence with equivalence classes of solutions, where equivalence is given by inner conjugation in $A\rtimes H$. Thus, classes in $\HG^1(K,A)$ correspond to $(A\rtimes H)$-extensions $M/K$ for which the restriction map $\Gal(M/K)\ra\Gal(L/K)$ coincides with the natural projection $A\rtimes H\ra H$.

\subsubsection*{Specialization} Suppose we have a finite $K(t)$-group $G$ over the rational function field $K(t)$, 
and a class  $\alpha\in \HG^1(K(t),G)$. 
If $G$ extends to an \'etale group scheme $\til G$ 
over the local ring $K[t]_{(t_0)}$ and if $\alpha$ is the image of a class $\til \alpha \in \HG^1(K[t]_{(t_0)}, \til G)$, then we let $\til G_{t_0}$ denote the specialization of the group scheme to $t_0$ and we call the image $\alpha_{t_0}$ of $\til \alpha$ in $\HG^1(K, \til G_{t_0})$ the {specialization} of $\alpha$ at $t_0$.

More generally, given a (smooth) $K$-scheme $Y$ and a $K$-group $G$, let $\HG^1(Y,G)$ denote the \`etale cohomolgy group. Torsors $P\ra Y$ correspond to classes in $\ZG^1(Y,G)$, cf.\ \ref{sec:twists}, and their specializations $P|_y$ at points $y\in Y(K)$ are classes in  $\ZG^1(K,G)$. The specialization map $|_{y}$  induces a map $\HG^1(Y,G)\ra \HG^1(K,G)$ on cohomology classes. 
\begin{defn}\label{def:param}
Let $G$ be an algebraic group over $K$, let $Y$ be a $K$-scheme, and  let $P_i\ra Y$, $i=1,\ldots,r$ be $G$-torsors. 1) We say that $P_i$ have {\it dense $K$-specializations in} $S\subseteq \HG^1(K,G)$ if for every (Zariski) dense open set $U \subset Y$ and every $G$-torsor $P_0\in S$, there exists $y \in U(K)$ such that $P_0 \cong P_i|_y$ for some $i$. \\  
2) Given a function field $F$ over $K$, we say $\gamma_1,\ldots,\gamma_r\in \HG^1(F,G)$ densely parametrize $S\subseteq \HG^1(K,G)$ if  there is an irreducible $K$-scheme $Y$ with function field $F$ such that $\gamma_1,\ldots,\gamma_r$ extend to classes over $Y$ with dense $K$-specializations in $\HG^1(K,G)$.  
\end{defn}

\noindent {\it Tori.} For an algebraic torus $T$ defined over $K$, let $\X(T)$ be the character group on $T$, that is, the $\Gamma_K$-module of all morphisms $f:T_{\sep K} \ra (\mG_m)_{\sep K}$ with the action $(\mbox{}^\sigma f)(x)=\sigma f(\sigma^{-1}x)$. 
Let $\X(T)^*=\Hom(\X(T),\mathbb Z)$ denote the dual lattice.
The dual torus, whose character lattice is the dual (cocharacter) lattice of $T$, is denoted $\hat T$.
More generally, for a (smooth diagonalizable) group scheme $G$,  write $X(G) = \Hom_{alg.gp}(G, \mathbb G_m)$ for the dual (character) group scheme. Note that for $G$ finite of order dividing $n$, we have $X(G) = \Hom_{alg.gp}(G, \mu_n)$.

For a finite separable extension $L/K$, and an algebraic torus $T$ defined over $L$, the (Weil) restriction of scalars $\Res_{L/K} T$ is a torus defined over $K$ for which  $\Res_{L/K} T(K') = T(L \otimes_K K')$ for every extension $K'/K$. In particular,  $(\Res_{L/K} T)_{\sep K} \cong  (T_{\sep K})^{[L:K]}$. In the case $T = \mG_{m, L}$, we then find $(\Res_{L/K} T)_{\sep K}$ is isomorphic to $\mathbb G_m^d$, where $d=[L:K]$, equipped with the action $\sigma (x_1,\ldots,x_d)= (\sigma(x_{\sigma^{-1}(1)}), \ldots,\sigma(x_{\sigma^{-1}(d)}))$. Thus,  $\X(\Res_{L/K} T)$ is a permutational module of rank $d$; that is, it is a free abelian group of rank $d$ on which $\Gamma_K$ acts by permuting the coordinates. A {\it quasitrivial} torus is a product of finitely many tori, each isomorphic to some restriction of scalars of $\mathbb G_m$. 

We let $\Res^1_{L/K}\mG_m$ denote the subtorus of $\Res_{L/K} \mG_m$ consisting of norm $1$ elements, and recall
that its lattice  $\X(\Res^1_{L/K} \mG_m)$ is naturally identified with the quotient of the permutation module ${\mathbb Z}^{[L:K]}$ by the image of the diagonal map $\mathbb Z \to \mathbb Z^{[L:K]}$, $n \mapsto (n, \ldots, n)$. In particular, if $[L:K]=2$,  the generator $\sigma$ of $\Gal(L/K)$ acts on $X(\Res^1_{L/K}\mG_m)\cong \Z$ by inversion $\sigma\cdot a = -a$, $a\in \Z$, so that $\sigma\cdot\alpha = \sigma(\alpha)^{-1}$ for $\alpha\in \Res^1_{L/K}\mG_m$. 



\subsection{Twisting and fibers in nonabelian Galois cohomology} \label{sec:twists}
Let us review the notion of twisting to understand the fibers of exact sequences in nonabelian Galois cohomology. For an \'etale $K$-group scheme $A$, we think of a (right) $A$-torsor as a generalization of a Galois extension of $K$ in the following sense. 
Right $A$-torsors over $K$ are in bijection with \'etale algebras $L/K$ with actions by the group scheme $A$. In the case $A$ is a constant group scheme, these are precisely $A$-extensions. 
Concretely, this correspondence works by associating to an $A$-extension $L/K$ its spectrum $\Spec L$. 
This spectrum carries 
a right action of $A$ on the points over the separable closure $\Spec L(\sep K) = \Hom_K(L, \sep K)$ by precomposition.

From now on, until we say otherwise, all torsors are considered as right torsors. Recall that an $A$-torsor $P$ is determined by $P(\sep K)$, the points of $P$ over a separable closure together with the (left) action of the absolute Galois group $\Gamma_K$, and together with an action of $A$ on the right. The latter amounts to a $\Gamma_K$-invariant action $P(\sep K) \times A(\sep K) \to P(\sep K)$ which is simply transitive. Choosing $x \in P(\sep K)$ gives rise to a map $\phi = \phi_{P, x}: \Gamma_K \to A(\sep K)$ determined by $\sigma(x) = x\phi(\sigma)$. We have
\[
x\phi(\sigma\tau) = \sigma(\tau(x)) = \sigma(x\phi(\tau) ) = \sigma(x)\sigma(\phi(\tau)) = x \phi(\sigma)\sigma(\phi(\tau)),
\] 
giving us $\phi(\sigma\tau) = \phi(\sigma)\sigma(\phi(\tau))$.
Denote by $\HG^1(K,A)$ the Galois cohomology pointed set $\HG^1(\Gamma_K,A(\sep K))=\Zd^1(\Gamma_K,A(\sep K))/\sim$, where the equivalence relation is given by an action of $A(\sep K)$ on crossed homomorphisms defined by $(a\cdot \phi)(\sigma) = a\phi(\sigma)\sigma(a)^{-1}$.  The elements of $\Zd^1(\Gamma_K, A(\sep K))$ are then in bijection with $A$-torsors $P$ together with a fixed identification $P(\sep K) = A(\sep K)$ corresponding to a choice of a $\sep K$-point, and the elements of $\HG^1(\Gamma_K, A(\sep K))$ are in bijection with isomorphism classes of $A$-torsors.

Recall that given an exact sequence of $K$-groups
$$
\xymatrix{
1 \ar[r] &  A\ar[r] &  B\ar[r] & C \ar[r] & 1,
}$$
by which we mean that we have an isomorphism of $A$ with a normal subgroup (which we think of as an identification $A \triangleleft B$) and an isomorphism $B/A \cong C$, we obtain a long exact sequence of pointed sets
\[
\xymatrix@C=.5cm @R=0.5cm{
1 \ar[r] & A(K) \ar[r] & B(K) \ar[r] & C(K) \ar[d] \\ 
& & & \HG^1(\Gamma_K, A(\sep K)) \ar[r] &
\HG^1(\Gamma_K, B(\sep K)) \ar[r] &
\HG^1(\Gamma_K, C(\sep K)) &
}
\]
In particular, if the map $B(K) \to C(K)$ is surjective, we can identify $\HG^1(\Gamma_K, A(\sep K))$ with the pointed kernel of $\HG^1(\Gamma_K, B(\sep K)) \to 
\HG^1(\Gamma_K, C(\sep K))$.

We now examine the fibers of the map $\HG^1(\Gamma_K, B(\sep K)) \to
\HG^1(\Gamma_K, C(\sep K))$ using the method of twisting, following \cite[\S III.1.6]{Gir}. In the context of the above exact sequence, given a $B$-torsor $P$, corresponding to a cocycle $\alpha \in \Zd^1(\Gamma_K, B)$, 
consider the image of the cocycle $\alpha$ in $\Zd^1(\Gamma_K, \Aut(B))$ with respect to the map $B \to \Aut(B)$ given by conjugation. 
By Galois descent, this image induces a $K$-group scheme $\sub\alpha B$. The left action of $\sub\alpha B$ and right action of $B$ on $P$ give $P$ the structure of a bitorsor. The operation $Q \mapsto Q \times_{\sub\alpha B} P$ of  contracted product with $P$ then gives a bijection between isomorphism classes of $\sub\alpha B$-torsors and $B$-torsors. Similarly, we may consider the image of $\alpha$ in $\Zd^1(\Gamma_K, \Aut(C))$ via the composition $B \to C \to \Aut(C)$ and the corresponding $K$-group scheme $\sub\alpha C$. Finally, as $A$ is normal in $B$,  consider the image of $\alpha$ in $\Zd^1(\Gamma, \Aut(A))$ given by the map $B \to \Aut(A)$ induced by conjugation, and the corresponding group scheme $\sub\alpha A$.
One may check that we thereby obtain a ``twisted'' exact sequence 
\[1 \ra  \sub\alpha A \ra \sub\alpha B \ra \sub\alpha C \ra 1,\]
and that contracted products induce a commutative diagram with vertical bijections:
\[\xymatrix{
\HG^1(\Gamma_K,  \sub\alpha A) \ar[r] & \HG^1(\Gamma_K, \sub\alpha B) \ar[r] \ar[d]^{\times_{\sub\alpha B} P} &
 \HG^1(\Gamma_K, \sub\alpha C) \ar[d]^{\times_{\sub\alpha C} P/A} 
\\
\HG^1(\Gamma_K, A) \ar[r] & \HG^1(\Gamma_K, B) \ar[r] & \HG^1(\Gamma_K, C).  
}\]
Here, $P/A$ denotes the contracted product $P \times_B (B/A) = P \times_B C$, which is a torsor corresponding to the image of the cochain $\alpha$ in $\Zd^1(\Gamma_K, C)$. In particular, given our $B$-torsor $P$, we may identify the fiber $F_{P/A}$ of the class of $P/A$ with respect to the map $\HG^1(\Gamma_K, B) \to \HG^1(\Gamma_K, C)$ with the pointed kernel of the map $\HG^1(\Gamma_K, \sub\alpha B) \to \HG^1(\Gamma_K, \sub\alpha C)$, which gives us an exact sequence of pointed sets:
\begin{equation}\label{equ:fiber-twist} \sub\alpha B(K) \to \sub\alpha C(K) \to \HG^1(\Gamma_K,  \sub\alpha A) \to F_{P/A} \to 1.
\end{equation}

\begin{example}\label{exam:twist-abelian}
Consider the special case that $B$ is a constant group and is expressed as a semidirect product $B = A \rtimes C$,  with $C$ abelian. Given a $C$-torsor $Q$, the splitting map $s:C \to B$ gives a $B$-torsor $P \cong Q \times_C B$ which maps to $Q$ (i.e. $P/A = P \times_B A \cong Q$). Let $\beta$ be a cocycle corresponding to $Q$. (We will also write $\beta$ for the image of this cocycle in $\Zd^1(\Gamma_K, B)$). As $C$ is abelian, we find that as conjugation by $C$ is trivial on $C$, we have $\sub\beta C(K) = C(K) = C$ and the $s$-image of $C$ in $B$ is likewise preserved by conjugation, showing $C = \sub\beta C(K) \subset \sub\beta B(K)$ which shows that $\sub\beta B(K) \to \sub\beta C(K)$ is surjective and hence we obtain an idenfitication $F_Q \cong \HG^1(\Gamma_K, \sub\beta A)$, where $\sub\beta A$ is the group scheme obtained by the image of $\beta$ in $\HG^1(\Gamma_K, \Aut(A))$ via the map $C \to \Aut(A)$ induced by conjugation.
\end{example}

\subsection{Embedding into tori and algebraic groups}\label{sec:embed} To study R-equivalence on Galois extensions, we follow the theory of $R$-equivalence on tori developed by Colliot-Th\'el\`ene--Sansuc \cite{CTS2}. This is a general reference for the  rest of this section. 
\newcommand{\mbb}{\mathbb}
\newcommand{\ov}{\overline}

We first note that every finite abelian $K$-group $A$ whose order is not divisible by the characteristic of $K$ embeds into a quasitrivial torus. Indeed, to do so, we may choose a projection $X_P \to X(A)$ with kernel $L$ a torsion-free abelian group (another lattice). This gives a short exact sequence
\begin{equation}\label{equ:quasi-trivial} 
\xymatrix{
1\ar[r] & A\ar[r] &  P \ar[r] & T\ar[r] & 1,
}
\end{equation}
for tori $T,P$ such that $P$ is quasitrivial,   $X(T) = L$ and $X(P)=X_P$.
Note that $\HG^1(K,P)=1$ for a quasitrivial torus $P$. Thus \eqref{equ:quasi-trivial} 
 yields  an epimorphism 
$T(K)\ra \HG^1(K,A).$

Analogously, for a finite $K$-groups $G$, we may choose a closed embedding $G \to \mathcal G$ for $\mathcal G=\GL_n,\SL_n$ or another rational connected algebraic group for which $\HG^1(K,\mathcal G)=1$.  
Consider the nonabelian short exact sequences $1 \to G \to \mathcal G \to \mathcal G / G \to 1$.
The long exact sequences (combined with the triviality of $\HG^1(K, \mathcal G)$) 
yield a surjective map $\delta:(\mathcal G/G)(K) \to \HG^1(K, G)$. 
Writing $X = \mathcal G/G$, we find that for any rational point $x \in X(K)$, this map fits into a commutative diagram of \'etale cohomology groups
\[\xymatrix@C=.5cm{
\HG^0(X, X) \ar@{=}[r] & \Hom(X, X) \ar[d] \ar[r] & \HG^1(X, G) \ar[d] \\
\HG^0(K, X) \ar@{=}[r] & X(K) \ar[r]^-{\delta} & \HG^1(K, G) \ar[r] & 0,
}\]
where the vertical arrows correspond to restricting to the rational point $x \in X(K)$. In particular, the identity map in $\Hom(X, X)$ maps to the right to a $G$-torsor $P$ over $X$, which, when restricted to the rational point $x$ gives a torsor $P|_x$ in the class $\delta(x)$. As $\delta$ is surjective, this implies that every $G$-torsor over $K$ (or in fact over any field) is obtained as a specialization of the torsor $P$.

\begin{lemma}\label{lem:dense}
The above torsor $P$ over $X$  admits dense $K$-specializations in $\HG^1(K,G)$.
\end{lemma}
\begin{proof}
Let $U$ be a dense open set in $X$ and suppose $P_0 \cong P|_x$ for some $x \in X$. It then remains to show that we can find some $x' \in U$ such that $P|_{x'} \cong P_0$. For this, we note that the long exact sequence in nonabelian Galois cohomology gives a bijection between $\HG^1(K, G)$ and the left $\mathcal G(K)$ orbits on $X(K)$. As the morphism $\mathcal G \to X$ given by $g \mapsto gx$ is surjective at the algebraic closure, it is dominant. Consequently, the inverse image of any open set $U$ is open in $\mathcal G$. As $\mathcal G$ is rational, its $K$-points are dense, and we may find some $g \in \mathcal G(K)$ such that $x' = gx \in U(K)$. But this means by our long exact sequence, that $[P|_{x'}] = \delta(x') = \delta(gx) = \delta(x) = [P|_x]$ giving us $P|_{x'} \cong P|_x$ as claimed.
\end{proof}
Denoting the set of $R$-equivalence classes on $\HG^1(K,G)$ and $T(K)$ by $\HG^1(K,G)/R$ and $T(K)/R$, resp.,  we  claim: 
\begin{lemma}\label{lem:BG-T}
Suppose we have a closed subgroup $G$ of a rational and connected linear algebraic group $\mathcal G$ over $K$ with $\HG^1(K, \mathcal G) = 1$. Let $X = \mathcal G/G$,  considered as a (pointed) $K$-scheme. Then $X$ admits a $G$-torsor $P\ra X$ with dense specializations, and the induced map $X(K) \to \HG^1(K, G)$ given by $x \mapsto P|_x$ induces a bijection $\HG^1(K,G)/R\cong X(K)/R$.
\end{lemma}
\begin{proof}
The construction of $P$ is given before Lemma~\ref{lem:dense} and the density of its specializations follows from that lemma.

Two $R$-equivalent points $P_0,P_1$ in $X(K)$ yield two extensions $\varphi_0,\varphi_1\in \HG^1(K,G)$ under $\delta$. A sequence of rational curves connecting $P_0$ and $P_1$, viewed as a sequence of elements in $X(K(t))$, yield the desired sequence of extensions $\psi_i(t)\in \HG^1(K(t),G)$, $i=1,\ldots,r$, such that $\psi_1(0)=\varphi_0$, $\psi_r(1)=\varphi_1$, and $\psi_{i+1}(0)=\psi_i(1)$ for $i=1,..,r-1$.  

For the converse, 
note that a sequence $\psi_i(t)\in \HG^1(K(t),G)$, $i=1,\ldots,r$ connecting $\varphi_0,\varphi_1$, as above, lifts to a sequence of elements in $X(K(t))$ and thus extends to a sequence of morphisms $U\ra X$ for some open subset $U\subseteq \mP^1$.  This sequence then connects two points $Q_0,Q_1\in X(K)$ whose images in $\HG^1(K,G)$ coincide with those of $P_0$ and $P_1$, resp. 
It therefore remains to show that any other point $P_i$ with the same image as  $Q_i$ is $R$-equivalent to it, $i=0,1$.  
Since $P_i$ and $Q_i$ have the same image in $\HG^1(K,G)$, the long exact sequence associated to \eqref{equ:quasi-trivial} 
shows that $P_i$ and $Q_i$ are in the same orbit under the action of $\mathcal G(K)$, say $P_i = g_i Q_i$ for some $g_i \in \mathcal G(K)$. However, since $\mathcal G$ is rational and connected, there is a map $\mP^1\ra \mathcal G$ connecting  $1\in \mathcal G(K)$ to $g_i$. 
Thus, the composition of this map with the map $\mathcal G \ra X$ and translation by $g_i$ gives a curve connecting $P_i$ and $Q_i$, $i=1,2$, as required. 
\end{proof}


\subsection{Flasque resolutions}\label{sec:flasque} To analyze the $R$-equivalence classes on $T$ as above, we use flasque resolutions. Let $L/K$ be a finite Galois extension with Galois group $G$ that splits $T$, that is, $T\cong \mG_m^d$ over $L$. 
For a torus $S$ which is split by $L$, one of the characterizations of $S$ being {\it flasque} is if its Tate cohomology group $\hat \HG^{-1}(H,\X(S))=1$ for every subgroup $H\leq G$. 
Recall that for a $G$-module $X$, the group $\HGH^{-1}(G,X)$ is the kernel of the map $N_{G}:X/I \ra X,$ where $I=\langle (\sigma-1)X\,|\,\sigma\in G\rangle$, and $N_G$ is induced from the $G$-norm $N_G(x)=\sum_{\sigma\in G}\sigma(x)$. 
In this context, Hilbert's Theorem 90 asserts that $\HGH^{-1}(C, X)=1$ for finite cyclic groups $C$ and permutation $C$-modules $X$.  

A {\it flasque} resolution for $T$ is a short exact sequence:
\begin{equation}\label{equ:flasque-res}
     \xymatrix{
1\ar[r] &  S \ar[r] & Q\ar[r] & T \ar[r] & 1,
}
\end{equation}
where $Q$ is a quasitrivial torus and $S$ is a flasque torus. We use the following  characterization of flasque resolutions from \cite[\S 2]{CTS2}. 
An exact sequence of tori \eqref{equ:flasque-res} with a Galois splitting extension $L/K$ with group $G$ is flasque if and only if the induced map $\pi_H:(\X(Q)^*)^H\ra (\X(T)^*)^H$ on $H$-invariants of the dual lattices is surjective for every $H\leq G$. 



The sequence \eqref{equ:flasque-res} induces a long exact sequence:
\begin{equation}\label{equ:long-flasque}
\xymatrix{
  Q(K) \ar[r] & T(K) \ar[r] & \HG^1(K,S) \ar[r] & 1.
}
\end{equation}
Since $Q$ is rational, two points of $Q(K)$ with the same image in $T(K)$ are directly $R$-equivalent.  Moreover, \cite[Thm.\  2]{CTS2} shows that the converse also holds, 
so that $T(K)/R\cong \HG^1(K,S)$.  
Combining this with Lemma \ref{lem:BG-T} one gets:
\begin{lemma}\label{lem:BG-r-equiv}
For a flasque resolution of tori as in \eqref{equ:flasque-res}, one has $T(K)/R \cong \HG^1(K,S)$. Thus, for a quasitrivial torus $P$ and a finite abelian $K$-group $A$ such that $P\ra T$ is 
an $A$-torsor (i.e. such that the kernel is $A$),  one has $\HG^1(K,A)/R\cong \HG^1(K,S)$.
\end{lemma}
Since $\HG^1(K,S)$ is finite for every flasque torus $S$ and finitely generated $K$ by \cite[Cor.\ 2]{CTS2}, it follows that:
\begin{cor}\label{cor:finiteness}
Let $K$ be a finitely-generated field and $A$ a  finite abelian $K$-group of order coprime to $\charak K$. Then the set $\HG^1(K,A)/R$ is finite. 
\end{cor}

Moreover, the map $\HG^1(K,A)\ra\HG^1(K,A)/R$ has the following interpretation, extending \cite[Prop.\ A.1]{CT2}: 
\begin{lemma}\label{lem:r-map}
    Given an embedding $A\ra P$ with quotient $T$ as in \eqref{equ:quasi-trivial}, and a flasque resolution $S\ra Q\ra T$ as in \eqref{equ:flasque-res}, there are natural maps $P\ra Q$ and $A\ra S$ such that:
    \begin{equation}\label{equ:lift}
\xymatrix{
1\ar[r] & A \ar[r] \ar[d]& P\ar[r] \ar[d]& T\ar[r] \ar@{=}[d] & 1 \\
1\ar[r] & S \ar[r] & Q\ar[r]^{} & T\ar[r] & 1.
}
\end{equation}
    is commutative. Moreover, upon identifying $\HG^1(K,S)$ with $\HG^1(K,A)/R$ through the isomorphism from Lemma \ref{lem:BG-r-equiv}, the  map $A\ra S$ induces a map $\HG^1(K,A)\ra \HG^1(K,S)$ which associates to each class its $R$-equivalence class. 
\end{lemma}
\begin{proof}
Let $G$ be the Galois group of a splitting extension for $A$ and the above tori. 
We first claim that the projection $P\ra T$ lifts to a projection $P\ra Q$, that is,  the following lifting problem is solvable:
$$\xymatrix{
 & P\ar[d] \ar@{-->}[dl] \\
Q \ar[r] & T.
}
$$
This is equivalent to showing that the map induced on dual (cocharacter) lattices $\X(P)^*\ra\X(T)^*$ extends to a map $\X(P)^*\ra\X(Q)^*$. The exact sequence
\[ 0 \to \X(S)^* \to \X(Q)^* \to \X(T)^* \to 0 \]
yields an exact sequence
\[ \Hom_{\Z[G]}(\X(P)^*, \X(Q)^*) \to \Hom_{\Z[G]}(\X(P)^*, \X(T)^*) \to \Ext^1_{\Z[G]}(\X(P)^*, \X(S)^*).\]
As $P$ is quasitrivial, $\X(P)^*$ is a direct summand of a sum of modules of the form $\Z[G/H]$ for subgroups $H \leq G$. 
It therefore follows that our map will extend if we can show $\Ext^1_{\Z[G]}(\Z[G/H], \X(S)^*) =0$ for all $H < G$. 
But Shapiro's lemma yields
\[
\Ext^1_{\Z[G]}(\Z[G/H], \X(S)^*) = 
\Ext_{\Z[H]}(\Z, \X(S)^*) = 
\HG^1(H, \X(S)^*) = 0,
\] 
since $S$ is flasque \cite[\S 1, Lem.\ 1]{CTS2}. Hence, our lifting problem is solvable as claimed.

As the composition $A \to P \to T$ is zero and coincides with the map $A \to P \to Q \to T$,  the image of $A$ in $Q$ lies in the kernel of $Q \to T$, so \eqref{equ:lift} is commutative. 

This yields a homomorphism between the long exact sequences corresponding to the rows of \eqref{equ:lift}. Since $P$ and $Q$ are quasitrivial, we get a commutative diagram: 
\begin{equation*}
\xymatrix{
P(K)\ar[r] \ar[d] & T(K)\ar[r] \ar[d] & \HG^1(K,A) \ar[r] \ar[d]^{\varphi} & 1 \\
Q(K)\ar[r] & T(K) \ar[r]^{} & \HG^1(K,S) \ar[r] & 1.
}
\end{equation*}
Since two points in $T(K)$ are $R$-equivalent if and only if they have the same image in $\HG^1(K,S)$ by Lemma \ref{lem:BG-r-equiv}, the following diagram is commutative
\begin{equation}\label{equ:r-equiv-maps}
\xymatrix{
T(K)\ar@{>>}[r] \ar[d]  & \HG^1(K,A)  \ar[d]^{\varphi}  \\
T(K)/R \ar[r]^{} & \HG^1(K,S),
}
\end{equation}
yielding the assertion.
\end{proof}
Furthermore, \eqref{equ:flasque-res} yields  rational parametrizing maps for $\HG^1(K,A)$:
\begin{cor}\label{cor:param}
    Let $A$ be a finite abelian $K$-group of order coprime to $\charak K$, and $r=\#\HG^1(K,A)/R$. Then $\HG^1(K,A)$ is densely parametrized by $r$ $A$-torsors $Q_i\ra Q$, $i=1,\ldots,r$ over a quasitrivial torus $Q$. 
\end{cor}
\begin{proof}
By the long exact sequence \eqref{equ:long-flasque}, 
 the image of the map $\pi:Q(K)\ra T(K)$ is of index $r$. Letting $t_1,\ldots,t_r\in T(K)$ denote $r$ distinct coset representatives,  $T(K)$ equals the union of the images of  
 $K$-points under the maps $\pi_i:Q\ra T$ given by 
 $p\mapsto t_i\cdot \pi(p)$, $i=1,\ldots, r$. Let $Q_i$ be the pullback (fiber product) of $P\ra T$ along $\pi_i$, $i=1,\ldots,r$. 

Pick $\alpha\in \HG^1(K,A)$. Since $T(K)\ra\HG^1(K,A)$ is surjective and associates to $t\in T(K)$ the fiber of $P\ra T$ above $t$ by the discussion before \cite[Prop.\ 36]{SerreGC}, there exists $t_0\in T(K)$ such that the residue of $P\ra T$ over $t_0$ is $\alpha$. By the first paragraph $t_0=\pi_i(q)$ for some $1\leq i\leq r$ and $q\in Q(K)$. Since $q$ is $K$-rational, the residue of $Q_i\ra Q$ at $q$ coincides with that of $P\ra T$ over $t$ and hence with $\alpha$
as required. 
\end{proof}

\subsection{Odd cyclic kernels}  It now follows from a result of Endo--Miyata that   $\HG^1(K,A)$ is $R$-trivial for an odd order cyclic (not necessarily constant) group $A$: 
\begin{prop}\label{prop:cyclic-odd}
    Let $C$ be finite $K$-group which is cyclic of odd order coprime to $\charak K$. Then $\HG^1(K,C)$ is densely parametrized by a $C$-torsor over a rational space, and in particular is $R$-trivial.
\end{prop}
\begin{proof}
Since $\HG^1(K,C)$ is isomorphic to the direct product $\prod \HG^1(K,C(p))$, where $C(p)$ is a $p$-Sylow subgroup of $C$, it suffices to prove the claim when $C$ is cyclic of prime-power order $p^k$. 
As $\Aut(C)$ is cyclic since $p$ is odd, the image of the action of $\Gamma_K$ on $C$ is cyclic, and hence the splitting field $L$ (fixed by the kernel of the action) is a cyclic extension of $K$. As in 
\eqref{equ:quasi-trivial},  consider a quasitrivial resolution 
$$
\xymatrix{
1 \ar[r] & C \ar[r] & P \ar[r] & T  \ar[r] & 1
}
$$
of $C$ in which  $T$ is split by $L$. 
Thus, the action of $\Gamma_K$ on $\X(T)$ factors through the cyclic and hence {\it Sylow-cyclic} group $\Gal(L/K)$. Here a group is Sylow-cyclic if its Sylow subgroups are cyclic. 

As in \eqref{equ:flasque-res}, consider a flasque resolution 
$$
\xymatrix{
1 \ar[r] & S \ar[r] & Q \ar[r] & T \ar[r] & 1.
}
$$ Since $\X(S)$ is flasque and is split by a Sylow-cyclic extension $L/K$, a result of Endo-Miyata \cite[Prop.\ 2]{CTS2} implies that $\X(S)$ is invertible, that is, $S$ is a direct factor of a quasitrivial torus. Thus,   $\HG^1(K,S)=1$ and  $\HG^1(K,C)$ is $R$-trivial by Lemma \ref{lem:BG-r-equiv}. Thus, Corollary \ref{cor:param} implies $\HG^1(K,C)$ is densely parametrized by a single $C$-torsor over a rational space. 
\end{proof}
To derive Theorem \ref{thm:semisimple}, we shall need  a version of the embedding \eqref{equ:quasi-trivial}  in families: 
\begin{lemma} \label{relative embedding}
   Let $A$ be a finite abelian  group scheme over a scheme $X$ whose order is invertible on $X$. Then there is an embedding $A \to P$ into a relative quasitrivial torus $P/X$. 
\end{lemma}
\begin{proof}
Without loss of generality, we may assume that $X$ is connected, and hence the geometric fibers of $A$ over $X$ are all of the same form, say $A_{\overline x} \cong \prod_i C_{n_i}$ for some fixed set of natural numbers $n_1, \ldots, n_r$.  
Let $\mathcal I$ be the \'etale sheaf on $X$ given by
\[\mathcal I(U) = \{\phi: A_U \to \prod_i \mu_{n_i, U} \mid \phi \text{ is an isomorphism}\}.\]
As $\mathcal I(U)$ is a sheaf-theoretic right $A$-torsor, it follows that it is representable by a finite scheme $Y$, \'etale over $X$. Further, we see that $A_Y \cong \prod_i \mu_{n_i, Y}$ by construction. In particular, via the embeddings $\mu_{n_i} \to \mbb G_m$, we can embed $A_Y$ into a split torus $T = \mbb G_{m, Y}^r$ and so we get an embedding of Weil restrictions $R_{Y/X} A \to R_{Y/X} T$. 
Combining this with the embedding given by the unit of the adjunction $A \to R_{Y/X} A_Y$ (this is a twisted form of the diagonal embedding $A \to A^N$), we obtain an embedding $A \to R_{Y/X}T$ of $A$ into a relative torus. But now by \cite[Prop.1.3(eq.~1.3.1)]{CTS3}, we find that we can embed $R_{Y/X}T$ into a relative quasitrivial torus, completing our proof.
\end{proof}

Theorem \ref{thm:semisimple} follows by repeating the proof of  Proposition \ref{prop:cyclic-odd} in families:
\begin{proof}[Proof of Theorem \ref{thm:semisimple}]
Suppose $\psi_0,\psi_1\in\HG^1(K,C\rtimes H)$  are classes whose images $\varphi_0,\varphi_1\in\HG^1(K,H)$ are $R$-equivalent. We claim that $\psi_0,\psi_1$ are $R$-equivalent as well. 
Since $\varphi_0,\varphi_1$ are $R$-equivalent, their lifts $\varphi_0^*, \varphi_1^*\in \HG^1(K,C\rtimes H)$ via the  given natural section $H\ra C\rtimes H$ are also $R$-equivalent.
Thus, it suffices to show $\varphi_i^*$ is $R$-equivalent to $\psi_i$, $i=0,1$. Since the fiber of $\HG^1(K,C\rtimes H)\ra \HG^1(K,H)$ over $\varphi_i$  coincides with $\HG^1(K,{}_{\varphi_i}C)$ as in \S\ref{sec:twists}, and $\HG^1(K,{}_{\varphi_i}C)$ is $R$-trivial by Proposition \ref{prop:cyclic-odd},  $\varphi_i^*$ and $\psi_i$ are $R$-equivalent, for $i=1,2$. 

To obtain a dense parametrization,  suppose $Y = \Spec(K[t_1, \ldots, t_m][f^{-1}])$ and suppose  $\phi \in \HG^1(Y, H)$  corresponds to a densely $Y$-parametrized family of $H$-torsors $\mathcal M \to Y$.
Say that a $(C \rtimes H)$-torsor $E/K$ lies over $\phi$ if the image $H$-torsor $E^C \in \HG^1(K, H)$ is the specialization of $\phi$ at a dense subset of $y \in Y(K)$. 

We  construct a rational space $\mathcal Q$ with a morphism $\pi: \mathcal Q \to Y$ and a  
($C \rtimes H$)-torsor $\mathcal E$ over $\mathcal Q$ such that for every ($C \rtimes H$)-torsor $E$ over $K$ which lies over $\phi$, there exists a dense subset of $q \in \mathcal Q(K)$ with $\mathcal E_q \cong E$. Since $\HG^1(K,H)$ is parametrized by $r$ rational spaces $Y$,  it would follow that so is $\HG^1(K,C\rtimes H)$. To construct $\mathcal Q$,  consider the group scheme $_\phi C$ over $Y$ which  embeds in a relative quasitrivial torus $\mathcal P$ over $Y$ with quotient $\mathcal P / _\phi C = \mathcal T$ by \cref{relative embedding}, 
yielding a short exact sequence of $Y$-group schemes: 
\[0 \to \mbox{}_\phi C \to \mathcal P \to \mathcal T \to 0.\]
In particular,  for any $y \in Y(K)$, Hilbert's theorem 90 gives the  surjectivity of the map: 
\begin{equation}\label{equ:surj-families}
\mathcal T_y(K) \twoheadrightarrow \HG^1(K, (\mbox{}_\phi C)_y)
\end{equation}
where $(\mbox{}_\phi C)_y = \mbox{}_{\phi_y} C$ is the fiber of the group scheme $\mbox{}_\phi C \to Y$ over $y$.
We may now construct a flasque resolution as in \cite[Prop.~1.3]{CTS3} for $\mathcal T$ to obtain a sequence
\[
0 \to \mathcal F \to \mathcal Q \to \mathcal T \to 0, 
\]
with relative quasitrivial tori $\mathcal Q,\mathcal F$, and  flasque fibers $\mathcal F_y$ for all $y\in Y(K)$. Moreover, as in the proof of Prop.\ \ref{prop:cyclic-odd}, since $C$ is cyclic of odd order, $\mathcal F_y$ is a direct factor of a quasitrivial torus and hence  
$\HG^1(K,\mathcal F_y)=1$ and $\mathcal Q_y(K)\to \mathcal T_y(K)$ is onto, for $y\in Y(K)$. Let $\pi:\mathcal Q\to Y$ be the associated $Y$-parametrization.

The pullback of $\mathcal P$ along the map $\mathcal Q\to\mathcal T$ induces a ${}_{\phi}C$-torsor over $\mathcal Q$.
Let $\mathcal E$ be the $(C\rtimes H)$-torsor corresponding to its image under  $\HG^1(\mathcal Q,{}_\phi C)\to \HG^1(\mathcal Q,{}_\phi(C\rtimes H))$. 
We claim that $\mathcal E\to\mathcal Q$ densely parametrizes the $(C\rtimes H)$-torsors $E/K$ lying over $\phi$. Indeed, recalling that $\mathcal M\to Y$  is a dense parametrization of a subset of $\HG^1(K,H)$, the image $E^C\in\HG^1(K,H)$
occurs as the fiber of $\mathcal M\to Y$ over a dense subset of $y\in Y(K)$. 
As in \eqref{equ:fiber-twist}, $E$ is then in the image of the map 
$\HG^1(K,({}_{\phi}C)_y)\to \HG^1(K,({}_\phi(C\rtimes H))_y)$. Since $\mathcal T_y(K)$ surjects onto $\HG^1(K,({}_{\phi}C)_y)$ by  \eqref{equ:surj-families}, and since 
$\mathcal Q_y(K)\to \mathcal T_y(K)$ is onto, we get that $\mathcal E/\mathcal Q$ in $\HG^1(\mathcal Q,(C\rtimes H))$ specializes to $E/K$ at a dense set of $K$-rational preimages in  $\pi^{-1}(y)\subseteq \mathcal Q$.
Thus  $\mathcal E\to \mathcal Q$ densely parametrizes $(C\rtimes H)$-torsors over $K$ which lie over $\phi$. Moreover, $\mathcal Q$ is rational since $Y$ is and the fibers of $\pi:\mathcal Q\to Y$ are, proving the claim. 
\end{proof}

\section{Torsors for cyclic $2$-groups}
In this section, we prove Theorems \ref{thm:2-power-cyclic} and \ref{thm:2-power}.

\subsection{Versal tori and their flasque resolutions} 
In this section, we will be interested in parametrizing and studying $R$-equivalence for torsors for certain finite group schemes which are twisted forms of 2-power order cyclic groups.
Let $K$ be a field of characteristic $\neq 2$ and $F=K[\sqrt{a}]$ for $a\in K^\times$. 
Suppose $C=C_{2^s}^{(a,\eps)}$ is the $\Gamma_K$-module $C_{2^s}=\langle c \mid c^{2^s}=1\rangle$ equipped with the action of $\Gal(F/K)$ via $\kappa\cdot c=c^\eps$, for the nontrivial $\kappa\in \Gal(F/K)$ and $\eps\in \{\pm 1\}$.  Let $\eta_s=\zeta_{2^s}+\zeta_{2^s}^{-1}$ 
and $E:=K(\eta_s)$. 
Set $s_0=[E:K]$, so that  $s_0\divides 2^{s-2}$. 
Let $L=E[\sqrt{-1},\sqrt{a}]$. Then $\Gal(L/K)=\langle\sigma,\tau,\kappa\rangle$, where
$\sigma,\kappa$ are the involutions fixing $E[\sqrt{a}], E[\sqrt{-1}]$, respectively\footnote{Note that $L$ is a ring which coincides with the field $F(\mu_{2^s})$ if $F$, $K(\eta_s)$, and $K[\sqrt{a}]$ are linearly disjoint fields over $K$.},  and $\tau$ is a generator of $\Gal(L/K[\sqrt{a},\sqrt{-1}])$.  
Let $M_{\eps}$ be the subring fixed by $\sigma^{(1-\eps)/2}\kappa$, that is, $M_1=E[\sqrt{-1}]$ is fixed by $\kappa$ and $M_{-1}=E[\sqrt{-a}]$ is fixed by $\sigma\kappa$. 
$$
\xymatrix@C=.1cm{
 & & L=K(\eta_s)[\sqrt{a}, \sqrt{-1}] \ar@{-}[drr]^{\kappa}\ar@{-}[d]^{\sigma}\ar@{-}[dll]_{\tau}& & \\
 K[\sqrt{a},\sqrt{-1}] \ar@{-}[dr] &  & K(\eta_s)[\sqrt{a}]   \ar@{-}[dl]   \ar@{-}[dr] & & M_1=K(\eta_s)[\sqrt{-1}]\ar@{-}[dl] \\
 & F=K(\sqrt a) & & E=K(\eta_s) & 
}
$$

Henceforth, to describe a torus $T$ over $K$, we describe $T(K')$ for every extension $K'/K$. 

\begin{lemma}\label{lem:embedding}
There is a $K$-embedding  of $C$ into the quasitrivial torus $P=\Res_{M_\eps/K}\mG_m$. For a finite extension $F'/F$ it sends  a generator $c\in C(F')$  to a primitive $2^s$-th root of unity $\zeta_{2^s}\in (L\otimes_K F')^\times$ via an identification $P(F')\cong (L\otimes_K F')^\times$. 
\end{lemma}
\begin{proof}
First, we identify $P(F)$ with $(M_\eps\otimes_K F)^\times$ so that it is isomorphic to $L^\times$ under the natural isomorphism  $\iota:M_\eps\otimes_K F\ra  L$. 
To show that the embedding $C\ra P(F)$, $c\mapsto \zeta_{2^s}\in L^\times$ is defined over $K$, it suffices to show that it preserves the $\kappa$-action (resp.\ $\Gal(F/K)$-action), that is,  $(1\otimes \kappa)$ maps $\iota^{-1}(\zeta_{2^s})$ to $\iota^{-1}(\zeta_{2^s})^\eps$.  

For $\eps=1$, $M_1=E[\sqrt{-1}]$ is fixed by $\kappa$ and hence its  image $ֿ\iota(M_1\otimes 1)$ under the isomorphism $\iota:M_1\otimes_K F\ra L$ is indeed fixed by $1\otimes\kappa$, as needed. 
For $\eps=-1$, the fixed subring of $1\otimes\kappa$ is $M_{-1}\otimes_KK$ which is mapped under $\iota$ to the fixed subring $M_{-1}=L^{\sigma\kappa}$. Thus, under $\iota$, the Galois action of $1\otimes \kappa$ on $M_{-1}\otimes_K F$ is equivalent to the Galois action of $\sigma\kappa$ on $L$.  Since $\sigma\kappa(\zeta_{2^s})=\zeta_{2^s}^{-1}\in L$ (as $\kappa$ acts trivially on $\zeta_{2^s}\in L$ and $\sigma$ is complex conjugation),  we get that $(1\otimes \kappa)(\iota^{-1}(\zeta_{2^s}))=\iota^{-1}(\zeta_{2^s})^{-1}$, as needed. 
\end{proof}
Henceforth, identify $C=\langle c\rangle$ as a group subscheme of $P=\Res_{M_\eps/K}\mG_m$ using the map from Lemma \ref{lem:embedding}. We now describe the quotient $P/C$. 
Write $\tau(\zeta_{2^s})=\zeta_{2^s}^m$ for $m$ of order $s_0$ mod $2^s$, so that $(m^{s_0}-1)/2^s$ is an odd number\footnote{This is equivalent to choosing $m$ of order $s_0$ mod $2^s$ and so that its order  mod $2^{s+1}$ is larger.}.

 Henceforth, abusing notation, we identify automorphisms $\sigma\in \Gal(L/K)$ with their corresponding automorphisms $1\otimes \sigma$ of $(K'\otimes_K L)/K$. Recall that $N_\sigma(x)=x\cdot\sigma(x)$.
\begin{lemma}\label{lem:quotient}
    Let $T$ be the torus over $K$ with: 
    $$T(K')\cong \{(p,q)\in (M_\eps\otimes_K K')^\times\times (E\otimes_K K')^\times \,|\,  N_\sigma(p)=\tau(q)/q\},$$ 
    for every finite extension $K'/K$. Then the surjection  $\pi_0:P\ra T$ given by:
    $$(M_\eps\otimes_K K')^\times \ni \beta\mapsto \left(\frac{\tau\beta}{\beta^m}N_\sigma(\beta)^{(m-1)/2}, N_\sigma(\beta)\right),$$
    has a cyclic kernel of order $m^{s_0}-1$ whose $2$-Sylow subgroup is $C$. 
\end{lemma}
\begin{proof}
Consider the torus $T'$ over $K$ given by 
$$T'(K')=\{(p,q)\in (M_\eps\otimes_K K')^\times\times (E\otimes_K K')^\times \,|\,  N_\sigma(p)=\tau(q)/q^m\}.$$ 
Consider the map $P(K')\ra T'(K')$ given by:
\begin{equation}\label{equ:map-to-T}
    (M_\eps\otimes_K K')^\times \ni \beta \mapsto \left(  \frac{\tau(\beta)}{\beta^m}, N_\sigma(\beta)\right).
\end{equation}  Its kernel consists of the $\beta\in P(K')$ such that $\sigma(\beta)=\beta^{-1}$, $\tau(\beta)=\beta^m$, and 
    as $\tau^{s_0}=1$,  one has $\beta^{m^{s_0}}=\tau^{s_0}(\beta)=\beta$, so that $\beta$ is an $(m^{s_0}-1)$-th root of unity. Setting $G:=\Gal(M_\eps/K)$,  we shall henceforth identify $P(\oline K)$ with the $G$-th power $\oline K^{G}$  via the map which sends $\beta\otimes 1$ 
    to a vector whose $\eta$-coordinate is $\eta(\beta)$ for every $\eta\in G$. 
    Via this identification $\sigma$ and $\tau$ (or more precisely, their images in $\Gal(M_\eps/K)$) act by permuting the coordinates and acting on the entries, that is, $\sigma\cdot (a_\eta)_{\eta\in G}=(\sigma(a_{\sigma^{-1}\eta}))_{\eta\in G}$ as  in Section \ref{sec:setup}. 
    Since $\sigma(\beta)=\beta^{-1}$ and $\tau(\beta)=\beta^{m}$ for $\beta$ in the kernel, the kernel consists of vectors $(a_\eta)_{\eta\in \Gal(M_\eps/K)}$ such that $a_1$ is a root of unity of order dividing $m^{s_0}-1$, and $a_{\sigma\eta} = \sigma(a_{\eta})^{-1}$ and $a_{\tau\eta}=\tau^{-1}(a_\eta)^{m}$ for every $\eta\in G$. Thus, the kernel is cyclic of order dividing $m^{s_0}-1$. 
    
    As this is also the action of $\sigma$ and $\tau$ on $C\subseteq P(F)$, we get that $C$ is contained in the kernel. 
    Since $2^s$ is the maximal $2$-power dividing $m^{s_0}-1$ by our choice of $m$, and since $C$ is contained in the kernel of the map, it follows that $C$ is the $2$-Sylow subgroup of this kernel. Finally, note that since $\dim P=\dim T'=4$ and our map $P\ra T'$ has a finite kernel, it is dominant and hence, as a finite dominant map between linear algebraic groups, surjective \cite[\S 4.2]{Hum}. 
    
     Finally, note that  the above torus $T'$ is isomorphic to $T$ via the change of variables map $T'(K')\ra T(K')$, $(p,q)\ra (p\cdot q^{(m-1)/2},q)$. Indeed, as $\sigma q=q$, one has $N_\sigma(q)=q^2$ for $q\in E\otimes K'$, and hence:
     $$N_\sigma(pq^{(m-1)/2})=N_\sigma(p)q^{m-1}=\frac{\tau(q)}{q^m}\cdot q^{m-1}=\frac{\tau{q}}{q},$$
     so that the image lies in $T$. Since the map is clearly invertible, it is an isomorphism. 
     Composing this isomorphism with the map \eqref{equ:map-to-T} gives the desired epimorphism. 
\end{proof}
    

We next take a flasque resolution of $T$. Let $Q$ be the quasi-split torus $\Res_{M_\eps/K}\mathbb G_m\times \mathbb G_m\times \Res_{M_\eps/K}\mathbb G_m$ so that $Q(K')\cong (M_\eps\otimes_K K')^\times \times K'^\times\times (M_\eps\otimes_K K')^\times$ for a finite extension $K'/K$. Consider the homomorphism $\pi:Q\ra T$ given on $K'$-points by 
$$(s,t,r)
\in (M_\eps\otimes_K K')^\times\times K'^\times\times (M_\eps\otimes_K K')^\times \mapsto \left(\frac{\tau(s)}{s}\frac{r}{\sigma(r)}, \frac{N_\sigma(s)}{t}\right).$$
It is straightforward to check that its image $(p,q)$ indeed satisfies $N_\sigma(p)=\tau(q)/q$. 
\begin{lemma}\label{lem:flasque}
The kernel $S:=\ker\pi$ is a flasque torus with $K'$-points $$S(K')\cong \left\{(s,r)\in (M_\eps\otimes_K K')^\times \times (M_\eps\otimes_K K')^\times\,|\, \frac{\sigma(r)}{r}=\frac{\tau(s)}{s}\right\}.$$
\end{lemma}
\begin{remark}\label{rem:ker}
The proof uses the following explicit form of the character lattice $\X(S)$: It is spanned by a basis $a,\tau a, \ldots,\tau^{s_0-1}a$, $b, \tau b,\ldots, \tau^{s_0-1}b,$ and $c$, where $a$ and $b$ are the projections to the $r$-coordinate and $s$-coordinate, resp., and $c$ is the map $(r,s)\mapsto t=N_\sigma(s)$. Setting $R:=\Z[\sigma,\tau]$, the resulting  $R$-action is as follows: 
 $\tau$ permutes $a,\tau a, \ldots,\tau^{s_0-1}a$ (resp.\ $b,\tau b, \ldots,\tau^{s_0-1}b$);  $\sigma b=-b+c$ (since $N_\sigma s=t$);   $\sigma a = a+\tau b-b$ (since $\sigma(r)/r=\tau(s)/s$); and $c$ is fixed by $R$: For,  $t=N_\sigma(s)$ is clearly invariant under $\sigma$ and it is $\tau$-invariant since $\tau(t)/t=N_\sigma(\tau(s)/s)=N_\sigma(\sigma(r)/r)=1$. 
\end{remark}
\begin{proof}[Proof of Lemma \ref{lem:flasque}]
By definition of $\pi$, the kernel is:
$$S(K')= \left\{(s,t,r)\in (M_\eps\otimes_K K')^\times \times K'^\times\times (M_\eps\otimes_K K')^\times\,|\,N_\sigma(s)=t, \, \frac{\sigma(r)}{r}=\frac{\tau(s)}{s}\right\},$$
and the isomorphism is obtained by forgetting $t=N_\sigma(s)$. 

We check next that $S$ is flasque by showing that $\HGH^{-1}(K,\X(S))=1$. To do so, we compute the kernels of the $H$-norm maps for all subgroups $H$ of $G_\eps:= \Gal(M_\eps/K)$. 

We start with cyclic groups $H$. For $H=\langle \tau^d\rangle$, where $d$ divides the order $s_0$ of $\tau$, the $H$-module $\X(S)$ splits as the direct sum of permutation $H$-modules and hence, in this case, the claim follows from Hilbert's Theorem 90, as in \S\ref{sec:flasque}. 

As this covers all subgroups in case $s_0=1$, henceforth assume $s_0\geq 2$. Next, we consider the other cyclic degree-$s_0$ subgroup  $H=\langle\sigma\tau\rangle$ and follow the description of $\X(S)$ from Remark \ref{rem:ker}. One has $N_H(c) = s_0\cdot c$,
\begin{align*}
    N_H(a) & = a+\sigma\tau a+\tau^2a+\cdots \sigma\tau^{s_0-1}a \\
    & = \sum_{i=0}^{s_0/2}(\tau^{2i}a) + \sum_{i=0}^{s_0/2}\left(\tau^{2i+1}a+\tau^{2i+1}(\tau-1)b\right)\\
    & = N_\tau(a)+N_{\tau^2}(\tau-1)\tau(b), \text{ and similarly:}\\
    N_H(b) & = N_{\tau^2}(1-\tau)b + \frac{s_0}{2}\cdot c.
\end{align*}
Thus, the kernel of $N_H$ is spanned by $b+\tau b-c, a-\tau a+\tau(b-\tau b),$ and their images under powers of $\tau$. 
These are principal elements since: $b+\tau b-c = (1-\sigma\tau)b$, $a-\tau a +\tau b-\tau^2b=(1-\sigma\tau)a$, and hence their images under powers of $\tau$ are also in $\Image{(1-\sigma\tau)}$, as claimed. 

Now consider the $G_\eps$-norm. One has $N_{G_\eps}(a)=2N_\tau(a)$, $N_{G_\eps}b=s_0\cdot c$, and $N_{G_\eps}c=2s_0\cdot c$. Thus, the kernel of $N_{G_\eps}$ is then spanned by $a-\tau a, 2b-c$ and their $\tau$-orbits. These are principal elements since $a-\tau(a)=(1-\tau)a$ and $2b-c=(1-\sigma)b$, and hence so are their $\tau$-orbit. 

Next note that the same arguments also give $\HGH^{-1}(G_\eps,\Ind_{s_0}^{ks_0} \X(S))=1$ where $\Ind_{s_0}^{ks_0} \X(S)$ is the module obtained by replacing $a$ and $b$ by $k$ copies $a_1,\ldots,a_k$, $b_1,\ldots,b_k$ of each. In other words,  $\Ind_{s_0}^{ks_0} \X(S)$ is the module spanned by $c$, $\tau^ja_i$ and $\tau^j b_i$, $i=1,\ldots,k$, $j=0,\ldots,s_0-1$,  where $c$ is fixed by $G_\eps$, $\tau$ permutes the $\tau$-orbits of $a_i$ and $b_i$, and $\sigma a_i=a_i+\tau b_i-b_i$, $\sigma b_i = -b_i+c$. 

Upon restricting to the subgroup $H=\langle\sigma,\tau^d \rangle$ for $d\divides s_0$, $d\neq s_0$, the $H$-module $\X(S)$ becomes isomorphic to $\Ind_{s_0/d}^{s_0} \X(S_{s_0/d})$, where $S_{s_0/d}$ is the torus obtained from $S$ by replacing $s_0$ by $s_0/d$. Hence upon replacing  $G_\eps$ by $H$ in the previous paragraph, we get
$\HGH^{-1}(H,\X(S))=\HGH^{-1}(H,\Ind_{s_0/d}^{s_0}\X(S_{s_0/d}))=1.$ Similarly for $H=\langle \sigma\tau^d\rangle$, we have  
$\HGH^{-1}(H,\X(S))=\HGH^{-1}(H,\Ind_{s_0/d}^{s_0}\X(S_{s_0/d}))=1$ upon replacing $\langle \sigma\tau\rangle$ by $\langle\sigma\tau^d\rangle$.  This covers all subgroups of $G_\eps$, proving that $S$ is flasque. 
\end{proof}
Finally, we deduce from the above:
\begin{cor}\label{cor:R-equiv}
    There is a bijection $\HG^1(K,C)/R\cong \HG^1(K,S)$. 
\end{cor}
\begin{proof}
    By Lemma  \ref{lem:quotient}, the kernel of the projection $\pi_0:P\ra T$ is a direct product of  $C$ with a  group scheme $C'$ over $K$ which is cyclic of odd order as an abstract group. Thus, by Lemma \ref{lem:BG-T},
    the $R$-equivalence classes on $\HG^1(K,\ker\pi_0)$  are in bijection with $T(K)/R$. However, since $\HG^1(K,\ker\pi_0)\cong \HG^1(K,C)\times \HG^1(K,C')$, and $\HG^1(K,C')$ is $R$-trivial by Proposition \ref{prop:cyclic-odd}, it follows that the $R$-equivalence classes on $\HG^1(K,C)$ are in bijection with those on $\HG^1(K,\ker\pi_0)$ and hence on $T(K)$.
    
    By Lemma \ref{lem:flasque}, the exact sequence
    $$ 
    \xymatrix{1 \ar[r] & S \ar[r] & Q \ar[r] & T \ar[r] & 1,}
    $$
    is a flasque resolution of $T$. Thus, $T(K)/R$ is in bijection with $\HG^1(K,S)$ by Lemma \ref{lem:BG-r-equiv}. Thus, by the above, the $R$-equivalence classes on $\HG^1(K,C)$ are also in bijection with $\HG^1(K,S)$.  
\end{proof}

\subsection{Cohomology of flasque tori} We follow the setup of the previous section to describe $\HG^1(K,S)$ explicitly as follows:
\begin{prop}\label{prop:r-equiv-cosets}
There is an injection $\HG^1(K,S)\ra E^\times/K^\times N_\sigma(M_\eps^\times)$ whose image consists of cosets of $e\in E^\times$ such that $\tau(e)/e\in N_\sigma(M_\eps^\times)$. 
\end{prop}
\begin{proof}
We use the notation of Lemma \ref{lem:flasque} and Remark \ref{rem:ker}. Let  $\hat B\leq \X(S)$ be the module spanned by $c$ and $\tau^ib$, $i=0,\ldots,s_0-1$; let $C$ be the module spanned by $c$; let $B$ (resp.\ $\hat A$) be the quotient $\hat B/C$ (resp.\ $\X(S)/C$). Let $A$ be the quotient $\hat A/B$. Thus we have the following commutative diagram of lattices with exact rows: 
 $$
\xymatrix{
 & & A \ar@{=}[r]& A & \\
0 \ar[r] & C \ar[r] & \X(S) \ar[r] \ar@{->>}[u] & \hat A \ar[r] \ar@{->>}[u] & 0 \\ 
0 \ar[r] & C \ar[r] \ar@{=}[u] & \hat B \ar[r] \ar@{^{(}->}[u] & B \ar[r] \ar@{^{(}->}[u] & 0. \\
 }
 $$ 
Letting $T_X$ denote the torus with character lattice $X$, we have:
$$
\xymatrix{
 &  T_A \ar@{=}[r] \ar@{^{(}->}[d] & T_A \ar@{^{(}->}[d] & & \\
1 \ar[r] & T_{\hat A} \ar[r] \ar@{->>}[d]_{p_y} & S \ar[r] \ar@{->>}[d]  & T_C \ar[r]  & 1 \\ 
1 \ar[r] & T_B \ar[r]  & T_{\hat B} \ar[r]  & T_C \ar[r]  \ar@{=}[u] & 1. \\
}
$$ 

Note that $T_A = \Res_{E/K}\mG_m$ so that $\HG^1(K,T_A)=0$, and that $T_C=\mG_m$ so that $\HG^1(K,T_C)=0$ by Hilbert 90. Since $T_C\cong\mathbb G_m$, the long exact sequences corresponding to the exact rows of the above diagram then yield: 
\begin{equation}\label{equ:cohom-diag}
\xymatrix{
\HG^0(K,T_C) = K^\times \ar[r] \ar@{=}[d] & \HG^1(K,T_{\hat A})  \ar[r] \ar@{^{(}->}[d]^{\varphi}  & \HG^1(K,S)  \ar[r] \ar@{^{(}->}[d] & 0 \\ 
\HG^0(K,T_C) = K^\times \ar[r]   & \HG^1(K,T_{B}) \ar[r]  & \HG^1(K,T_{\hat B}) \ar[r]  & 0. \\
}
\end{equation}
Thus $\HG^1(K,S)$ embeds into $\HG^1(K,T_{\hat B})$ with image isomorphic to the quotient of the $\varphi$-image of $\HG^1(K,T_{\hat A})$ by $K^\times$. We claim that $\HG^1(K,T_{\hat B})\cong E^\times/K^\times N_\sigma(M_\eps)^\times$: 

Note that $B$ is spanned by the images of $\tau^ib$, $i=0,\ldots,s_0-1$, and $\sigma$ acts on these images by $\sigma \tau^ib=-\tau^ib$  in view of Remark \ref{rem:ker}. Thus,  $T_B\cong \Res_{E/K}(\Res^1_{M_\eps/E}\mG_m)$ as in \S\ref{sec:setup}. 
Shapiro's lemma yields an isomorphism 
$\HG^1(K,T_B)\cong 
\HG^1(E,\Res^1_{M_\eps/E}\mG_m).$
Since $\HG^1(E,\Res_{M_\eps/E}\mG_m)=0$,  
the short exact sequence 
 $$1\ra \Res^1_{M_\eps/E}\mG_m\ra \Res_{M_\eps/E}\mG_m\stackrel{N_\sigma}{\ra} 
 \mG_m \ra 1,$$
then gives the isomorphism  $\HG^1(E,\Res^1_{M_\eps/E}\mG_m)\cong E^\times /N_\sigma(M_\eps^\times)$. 
In total,  $\HG^1(K,T_B)\cong E^\times/N_\sigma(M_\eps)^\times$ and hence $\HG^1(K,T_{\hat B})\cong E^\times/K^\times N_\sigma(M_\eps)^\times$ by  \eqref{equ:cohom-diag}, yielding the claim. 

It remains to find the image of the projection $(p_y)_*:\HG^1(K, T_{\hat A})\ra \HG^1(K, T_{B})$ induced from the projection $p_y:T_{\hat A}\ra T_B$. 
By Remark \ref{rem:ker}, $\hat A$ is spanned by $\tau^ia,\tau^ib$, $i=0,\ldots,s_0-1$ subject to the relations $\sigma a -a=\tau b -b$ and $\sigma b+b=0$, so that:
$$
T_{\hat A}(K')=\{(x,y)\in (M_\eps\otimes_K K')\times (M_\eps\otimes_K K')\,|\,\frac{\sigma(x)}{x}=\frac{\tau(y)}{y},N_\sigma(y)=1\},
$$
for every $K'\supseteq K$. Note that $p_y(x,y)=y$. 
Consider also the projection $p_x:T_{\hat A}\ra R_{M_\eps/K}\mG_m$,  $(x,y)\mapsto x$ to the $x$-coordinate. 
Since $\sigma x/x = \tau y/y$ in $T_{\hat A}$, we see that the composition of $p_x$ with the map  $(\sigma-1):R_{M_\eps/K}\mG_m\ra T_B$, $x\mapsto \sigma(x)/x$ coincides with the composition of $p_y:T_{\hat A}\ra T_B$, $(x,y)\mapsto y$ with the maps $(\tau-1):T_B\ra T_B$, $y\mapsto \tau y/y$. In addition the kernel of $\sigma-1$ consists of $x\in (M_\eps\otimes_K K')^\times$ such that $\sigma x=x$, that is, of $x\in (E\otimes_K K')^\times$. Thus $\ker(\sigma-1)=R_{E/K}\mG_m=T_A$. In total, the following diagram is commutative with exact rows:
\begin{equation*}
    \xymatrix{
    1\ar[r] & T_A \ar[r] \ar@{=}[d] & T_{\hat A} \ar[r]^{p_y} \ar[d]_{p_x}& T_B \ar[r] \ar[d]^{\tau-1} & 1 \\
    1\ar[r] & T_A \ar[r] & R_{M_\eps/K}\mG_m \ar[r]^>>>>>{\sigma-1} & T_B \ar[r] & 1.
    }
\end{equation*}
Since $R_{M_\eps/K}\mG_m$ is quasitrivial, the above short exact sequences yield the following maps between the corresponding long exact sequences:
\begin{equation}
    \xymatrix{
    \HG^1(K,T_{\hat A}) \ar[r]^{(p_x)_*} & \HG^1(K,T_B) \ar[r] \ar[d]_{(\tau-1)_*} & \HG^2(K, T_A) \ar@{=}[d] &  \\
    1 \ar[r] & \HG^1(K,T_B) \ar[r] & \HG^2(K,T_A) \ar[r] & 1. 
    }
\end{equation}
Since the map $\HG^1(K,T_B)\ra\HG^2(K,T_A)$ on the bottom row is an isomorphism, it follows that  the image of $(p_x)_*:\HG^1(K,T_{\hat A})\ra \HG^1(K,T_B)$ coincides with the kernel of  $(\tau-1)_*$. Identifying $\HG^1(K,T_B)$ with $E^\times/N_\sigma(M_\eps^\times)$ as above, $(\tau-1)_*$  yields the map 
$E^\times/N_\sigma(M_\eps^\times)\ra E^\times/N_\sigma(M_\eps^\times)$, $y\mapsto \tau y/y$. Its kernel consists of those $\alpha\in E^\times/N_\sigma(M_\eps)^\times$ such that $\tau(\alpha)/\alpha\in N_\sigma(M_\eps)^\times$, 
and the claim follows from the consequence of \eqref{equ:cohom-diag} (appearing right after it). 
\end{proof}%
The subgroup $S_K=S_K(\eps,a)$ of $E^\times/K^\times N_\sigma(M_\eps^\times)$ consisting of cosets of $e\in E^\times$ such that $\tau(e)/e\in N_\sigma(M_\eps^\times)$ is also known as the group of special projective conorms, a terminology due to Platonov. 
The following proposition describes $S_K$ in terms of the Brauer group. It extends \cite[Thm.\ 3.2]{Schneps} of Martinais--Schneps (who consider the case $M_\eps=K(\mu_{2^s})$).
\begin{prop}\label{prop:projective} 
For $a\in K$ and $\eps\in \{\pm 1\}$, let $N_\eps=K[\sqrt{-1}]$ if $\eps=1$ and $N_\eps=K[\sqrt{-a}]$ if $\eps=-1$. Then:
$$S_K\cong \coker(\Br(E/K)+\Br(N_\eps/K)\ra \Br(M_\eps/K)).$$ 
\end{prop}

Our proof  is somewhat different from \cite[Thm.\ 3.2]{Schneps} and rests on:
\begin{lemma}\label{lem:spectral}
Suppose $M/K$ is a Galois (\'etale algebra) extension of $K$ with $\Gal(M/K) = G$ and $H \triangleleft G$ is a  normal subgroup with cyclic quotient $C:=G/H$. Let $E = M^H$.
$\Br(M/K) \to \Br(M/E)^C$
is surjective.
\end{lemma}
\begin{proof}
We consider the Lyndon-Hochschild-Serre spectral sequence 
\[E^{p,q}_2 = \HG^p(C, \HG^q(H, M^*)) \Longrightarrow \HG^{p + q}(G, M^*) = E^{p + q}\]
which gives an exact sequence associated to terms on the $E_3$ page:
\[\HG^2(G, M^*) = E^2 \to E^{0, 2}_3 \to E^{3, 0}_3 \to E^3.\]
By definition, we may identify
\[  E^{0,2}_3 = \ker\left[ E^{0,2}_2 \to E^{2, 1}_2 \right] = \ker\left[ \HG^0(C, \HG^2(H, M^*)) \to \HG^2(C, \HG^1(H, M^*)) \right], \]
and as $\HG^1(H, M^*) = 0$ by Hilbert 90,  $E^{0,2}_3$ can be identified with $\Br(M/E)^C$. We also find that $E^{3, 0}_3 = \HG^3(C, E^*) = \HG^1(C, E^*) = 0$, using the fact that $C$ is cyclic, and again using Hilbert 90. Putting these together, we obtain the desired exact sequence
\[\Br(M/K) \to \Br(M/E)^C \to 0. \]
\end{proof}
\begin{proof}[Proof of Proposition \ref{prop:projective}]
First note that  $L/E$ is a biquadratic extension which is defined over $K$, that is, $K[\sqrt{a},\sqrt{-1}]\otimes_K E\cong L$ for the biquadratic extension $K[\sqrt{a},\sqrt{-1}]/K$. The extension $M_\eps/E$ is also defined over $K$ by the quadratic extension $N_\eps/K$ defined above, that is, $N_\eps\otimes_K E\cong M_\eps$.

Now consider the restriction map $\Br(M_\eps/K)\ra \Br(M_\eps/E).$ 
Since $E/K$ is cyclic, Lemma \ref{lem:spectral} implies that its image is the  subgroup $\Br(M_\eps/E)^\tau$ of $\tau$-invariants in $\Br(M_\eps/E)$. As $M_\eps/E$ is cyclic,  $\Br(M_\eps/E)\cong E^\times/N_{M_\eps/E}(M_\eps)^\times$, and the image $H$ of $\Br(M_\eps/E)^\tau$ under this isomorphism consists of the cosets of $z\in E^\times$ for which the cyclic algebras $(M_\eps/E,\sigma,z)$ and $(M_\eps/E,\sigma,z^\tau)$ are isomorphic, or equivalently such that $z^\tau/z\in N_{M_\eps/E}(M_\eps^\times)$. 

Writing $N(M_\eps)$ for $N_{M_\eps/K}(M_\eps)^\times$, we note that  $S_K\cong H/(K^\times N(M_\eps)/N(M_\eps))$. Thus, the kernel of the composed map $\Br(M_\eps/K)\ra H\ra S_K$  consists of those classes whose restriction to $E$ is equivalent to a cyclic algebra $(M_\eps/E,\sigma,z')$ for some $z'\in K^\times$. This is clearly the case for every decomposed algebra $(E/K,\tau,y)\otimes_K (N_\eps/K,\sigma',z')$, where  $y,z'\in K^\times$ and $\sigma'$ is the nontrivial automorphism in $\Gal(N_\eps/K)$. Conversely, for every class $\alpha\in \Br(M_\eps/K)$ whose restriction to $E$ takes the form $(M_\eps/E,\sigma,z')$, $z'\in K^\times$, the restriction of the class $\alpha-[(N_\eps/K,\sigma',z')]\in \Br(M_\eps/K)$ to $E$ is trivial. As $E/K$ is cyclic, this class is equivalent to $[(E/K,\tau,y)]$ for some $y\in K^\times$, so that $\alpha$ is the class of a decomposed algebra $(E/K,\tau,y)\otimes_K (N_\eps/K,\sigma',z')$, where  $y,z'\in K^\times$. It follows that $S_K$ is isomorphic to the quotient of $\Br(M_\eps/K)$ by the image of $\Br(E/K)+\Br(N_\eps/K)$ as desired. 
\end{proof}
\begin{proof}[Proof of   Theorems \ref{thm:2-power-cyclic}.\ and \ref{thm:2-power}] The first assertion of Theorem \ref{thm:2-power-cyclic} (resp.\ \ref{thm:2-power}) is the combination of  Propositions \ref{prop:r-equiv-cosets} and \ref{prop:projective} for $\eps=1$ (resp.\ $\eps=-1$). There  cannot be a  parametrizing set of cardinality $\leq r-1$ since there are at least $r$ $R$-equivalence classes by the first assertion. The existence of a densely parametrizing set with $r$ torsors over a rational space then follows from Corollary \ref{cor:param}.
\end{proof}
\begin{remark}\label{rem:noncyclic}
1) Note that if $M_\eps/K$ is not a field extension, then either $E=K$, or $K[\sqrt{\eps a}]/K$ is a trivial extension, or $K[\sqrt{\eps a}]\subseteq E$. In all these cases, the resulting group $S_K$ is clearly trivial. \\
2) We note that the treatment in this section applies similarly to the action $\kappa\cdot c=c^\eps$, where $-1\neq \eps\in (\Z/2^s)^\times$ is an element of order $2$ generating a direct summand of $(\Z/2^s)^\times$. To do so, one needs to replace $\sigma$ so that $\sigma(z)=z^\eps$, and replace $E=K(\eta_s)$ by the new fixed field  $E=K(\sqrt{-1}\cdot \eta_s)$ of $\sigma$. A study of the other possible actions is of further interest. \\
3) In case $K$ is a number field, the subgroup $\Br(E/K)+\Br(N_\eps/K)$ coincides with the decomposable subgroup $\Dec(M_\eps/K)$, so that the quotient represents the indecomposable part of that relative Brauer group. 
\end{remark}

\section{Dihedral extensions} \label{sec:dihedal}

In this section, we consider  $R$-equivalence  on dihedral extensions over a number field $K$, proving 
Theorem \ref{thm:main}. 

The proof uses the following description of the group $S_K$ from Proposition \ref{prop:projective}. For a Galois field extension $M/K$ and a place 
$\nu$ of $K$, let $M_\nu$ denote the completion\footnote{Note that the completion $M_{\fP}$ is independent of the choice of prime $\fP$ over $\fp$.} of $M$ at a place $\omega$ of $M$ lying over $\nu$. Let $\Inv_\nu:\Br(K_\nu)\ra\mQ/\Z$ denote the Hasse-invariant map. 
\begin{lemma}\label{lem:local}
    Let $M/K$ be a Galois field extension of group $C_2\times C_{s_0}$ for a nontrivial $2$-power $s_0$, and set $E=M^{C_2}$ and $N:=M^{C_{s_0}}$. 
    Let $\mathcal S$ be the set of primes $\fp$ of $K$ for which
    $M_\fp/K_\fp$ is noncyclic, 
    and $\mathcal S_{f}\subseteq \mathcal S$ its subset of primes $\fp$ with full local degree $[M_\fp:K_\fp]=[M:K]$. 
    Then the kernel of the map 
    $$I:\Br(M/K)\ra (\frac{1}{2}\Z/\Z)^{\mathcal S}, \alpha\mapsto \left(\frac{[M_\fp:K_\fp]}{2}\cdot \Inv_\fp(\alpha)\right)_{\fp\in \mathcal S},$$ is $\Br(E/K)+\Br(N/K)$ and its image consists of the elements $(x_\fp)_{\fp}\in (\frac{1}{2}\Z/\Z)^{\mathcal S}$ with $\sum_{\fp\in \mathcal S_f}x_\fp=0$. 
    In particular,  the quotient of $\Br(M/K)$ by $\Br(E/K)+\Br(N/K)$ is isomorphic to this image. 
\end{lemma}
Recall that by the Albert--Brauer--Hasse--Noether (ABHN) theorem,   the restriction map $\Br(K)\ra \bigoplus_{\nu} \Br(K_\nu)$, where $\nu$ runs over all places of $K$,  maps $\Br(K)$ isomorphically to the subgroup of elements $(\alpha_\nu)_\nu\in \bigoplus_\nu\Br(K_\nu)$ whose sum of invariants $\sum_{\nu}\Inv_\nu(\alpha_\nu)$ is trivial. 
We call the set of places $\nu$ at which $\alpha$ has a nontrivial invariant  the {\it support} of $\alpha$. Note that $\Inv_\nu$ is an isomorphism if $\nu$ is finite, is the trivial map if $\nu$ is complex, and an injection to $\frac{1}{2}\Z/\Z$ if $\nu$ is real. Moreover for a Galois $M/K$, $\Br(M/K)$ consists of those classes $\alpha$ for which $\alpha_\nu$ is split by $M_\nu$ or equivalently $[M_\nu:K_\nu]\cdot\Inv_\nu(\alpha)=0$, see \cite[\S 31-32]{Rei}. Over $p$-adic fields one has:
\begin{remark}\label{rem:p-adic}
In the setup of Lemma \ref{lem:local}, the quotient of $\Br(M_\fp/K_\fp)$ by $\Br(E_\fp/K_\fp)+\Br(N_\fp/K_\fp)$ is nontrivial if and only if $\Gal(M_\fp/K_\fp)$ is noncyclic: For, the latter occurs if and only if $[M_\fp:K_\fp]$ is a strictly larger $2$-power than the degree of any cyclic subextension of $M_\fp/K_\fp$, and hence as in ABHN, if and only if classes $\alpha_\fp$ with invariants of order $[M_\fp:K_\fp]$ 
do not split under cyclic subextensions of $M_\fp/K_\fp$. 
\end{remark}
\begin{proof}[Proof of Lemma \ref{lem:local}] Set $s_0:=[E:K]$ and $D=D(M/K):=\Br(E/K)+\Br(N/K)$ the subgroup split by cyclic subextensions. 
Since  for each $\nu\notin \mathcal S_f$ the invariant of an element in $\Br(M/K)$ at $\nu$ must divide $s_0=[M:K]/2$, one has $\sum_{\fp\in \mathcal S_f}\Inv_\fp(\alpha)\in \frac{1}{s_0}\Z/\Z$.  Thus, the image of $I$ is contained in the  subgroup $I_f\leq (\frac{1}{2}\Z/\Z)^{\mathcal S}$ of elements $(x_\fp)_{\fp\in\mathcal S}$ such that $\sum_{\fp\in\mathcal S_f}x_\fp=0$.  We claim that  the image of $I$ equals $I_f$. 
 
 By Chebotarev's density theorem there exists a prime $\fq$ of $K$, not in $\mathcal S$, that is unramified in $M/K$, and for which $[M_\fq:K_\fq]=[E_{\fq}:K_{\fq}]=[E:K]=s_0$.  Consider  
  classes $\alpha_{\mathfrak{r}}\in\Br(K)$, $\fr\notin\mathcal S_f$ with:
$$
\Inv_{\fp}(\alpha_{\mathfrak r}) = \left\{\begin{array}{cl} 1/[M_\fr:K_\fr] & \text{if }\fp=\mathfrak{r} \\ -1/[M_\fr:K_\fr] & \text{if }\fp=\fq \\
0 & \text{otherwise,}  \end{array} \right.
$$ 
that exist by ABHN. 
Since 
$[M_{\mathfrak{r}}:K_{\mathfrak{r}}]\divides s_0=[M_\fq:K_\fq]$ for such $\mathfrak{r}$, we have $\alpha_\fr\in \Br(M/K)$. Moreover for $\fr\in\mathcal S\setminus\mathcal S_f$, we have $\alpha_\fr\notin D$ since $\alpha_r$ is not split by any cyclic subextension. 
Similarly, if $\mathcal S_f\neq\emptyset$,  pick $\fq'\in \mathcal S_f$ and  define for  $\fr\in \mathcal S_f\setminus\{\fq'\}$:
$$
\Inv_{\fp}(\alpha_{\mathfrak r}) = \left\{\begin{array}{cl} 1/[M:K] & \text{if }\fp=\mathfrak{r} \\ -1/[M:K] & \text{if }\fp=\fq' \\
0 & \text{otherwise.}  \end{array} \right.
$$
Since $\fq',\fr\in \mathcal S_f$, we have $\alpha_\fr\in \Br(M/K)$ and $\alpha_\fr\not\in D$.  Moreover, since $\alpha_\fr$, $\fr\in\mathcal S$ are supported at distinct primes, their images under $I$ generate a subgroup of $I_f$ of rank $\#\mathcal S-1$ if $\mathcal S_f\neq \emptyset$, resp.\ $\#\mathcal S$ if $\mathcal S_f=\emptyset$. Hence the image is $I_f$.


It remains to show $\ker I=D$. 
The containment $D\subseteq \ker I$ holds since   $([M_\fp:K_\fp]/2)\cdot \Inv_\fp(\alpha)=0$ for  $\alpha\in D$ and $\fp\in \mathcal S$. 
For the reverse inclusion, note that since $M_\fr/K_\fr$ has Galois group $C_2\times C$ for cyclic $C$, using an argument similar to the above (and using the additional prime  $\fq$) ABHN implies that every $\alpha\in \ker I$ is the sum of an element in $D$ with an element in $\Br(M/K)$ whose invariants at $\mathcal S$ are trivial. However,  the subgroup of elements with trivial invariants at $\mathcal S$ is generated by the elements $\alpha_\fr$, $\fr\notin\mathcal S$ and these are also contained in $D$, so that $\ker I=D$. 
\end{proof}
The proof of the theorem  uses  the following fact: For a function field $F$ of a genus $0$ curve over $K$, the kernel of local-to-global map $\Br(F)\ra\bigoplus_\nu \HG^1(F_\nu,\mQ/\Z)$ consists of the images of the constant Brauer classes $\Br(K)$, where $\nu$ runs over places of $F$ that are trivial on $K$. This fact follows from the exact sequence \cite[Prop.\ 5.4.2]{CTSK} as follows. Choosing a smooth projective model $X$ of $F$, the kernel of the map is the unramified Brauer group $\Br(X)$, and by Tsen's theorem, it coincides with its subgroup $\Br_1(X)$ consisting of elements split by $\oline K$. Let $X^s$ denote the base change of $X$ to $\oline K$. The exact sequence implies that the kernel of the map $\Br_1(X)\ra\HG^1(K,\Pic(X^s))$ coincides with the image of $\Br(K)\ra \Br_1(X)$ and hence consists of the constants classes. However, since $X$ is a smooth projective curve, $\Pic(X^s)\cong \Z$ and $\HG^1(K,\Pic(X^s))$ is trivial, so that that kernel is all of $\Br_1(X)$. 

\begin{proof}[Proof of Theorem \ref{thm:main}]
Set $G=D_{2^s}=\langle r,s\,|\,r^{2^s}=s^2=srs=1\rangle$ and let $\pi:G\ra C_2$ be the projection mod $\langle r\rangle$. 
We claim that every  $\beta\in \HG^1(K,G)$   is $R$-equivalent to the trivial element.
Let $\gamma$ denote the image of $\beta$ under the map $\pi_*:\HG^1(K,G)\ra \HG^1(K,C_2)$ induced from $\pi$, and $K[\sqrt{a}]/K$ the quadratic \'etale algebra corresponding to $\gamma$. As in \S \ref{sec:twists}, the fiber of $\pi_*$ over $\gamma$ is $\HG^1(K,C_{2^s}^{(a)})$, where $C_{2^s}^{(a)}$ is the twist of $C_{2^s}$ by $\gamma$. Let $\beta'\in \HG^1(K,C_{2^s}^{(a)})$ represent the preimage of $\beta$ in this fiber.


As in Lemma \ref{lem:embedding}, $C_{2^s}^{(a)}$ 
embeds into a quasi trivial torus $P^{(a)}$ on which 
$\Gamma_K$ acts through $\Gal(M/K)$, where 
$M=M^{(a)}:=E[\sqrt{-a}]$ and $E=K(\eta_{s})$. The 
quotient torus $T^{(a)}=P^{(a)}/C_{2^s}^{(a)}$ has a flasque resolution $S^{(a)}\ra Q^{(a)}\ra T^{(a)}$ as in Lemma \ref{lem:flasque}. 
The composition of the map $\HG^1(K,C_{2^s}^{(a)})\ra \HG^1(K,S^{(a)})$ from \eqref{equ:r-equiv-maps} with the maps from Propositions \ref{prop:r-equiv-cosets} and \ref{prop:projective} yields a map $\psi^{(a)}$ from $\HG^1(K,C_{2^s}^{(a)})$  
to the quotient $\oline{\Br}(M/K)$ of  $\Br(M/K)$ by $\Br(E/K)+\Br(K[\sqrt{-a}]/K)$. 
Thus, $\alpha':=\psi^{(a)}(\beta')$ corresponds to the $R$-equivalence class of $\beta'$ in $\HG^1(K,C_{2^s}^{(a)})$. Let $\alpha\in \Br(M/K)$ be a representative of $\alpha'$. 

We first prove the assertion when $M/E$ is not a field extension. In such a case,  either $E=K$, or  $K[\sqrt{-a}]/K$ is a split extension, or $K[\sqrt{-a}]\subseteq E$. In all these cases $\oline{\Br}(M/K)$ is trivial as in Remark \ref{rem:noncyclic}, and hence $\beta'$  is $R$-equivalent to the trivial element in $0\in \HG^1(K,C_{2^s}^{(a)})$. Thus, their images $\beta$ and $\gamma'$ in $\HG^1(K,G)$ are also $R$-equivalent. Finally, note that since all elements in $\HG^1(K,C_2)$ are $R$-equivalent, so are their images under the splitting $\HG^1(K,C_2)\ra\HG^1(K,G)$ which associates to an element the trivial element in its fiber. It follows that $\gamma'$ is $R$-equivalent to the trivial element in $\HG^1(K,G)$ and hence so is $\beta$, yielding the claim. 

Henceforth, assume $M/K$ is a noncyclic  field extension. We shall define $p(t)\in K(t)$ so that $p(0)=-1$, $p(1)=a$, and  an element $\alpha(t)\in {\Br}(M^{(p(t))}/K(t))$ split by $K(t,\sqrt{-p(t)})$ so that $\alpha(0)\sim 0$ but $\alpha(1) \sim \alpha$. Then, picking   $\chi \in \HG^1(K(t),C_{2^s}^{(p(t))})$ whose $R$-equivalence class $\psi^{(p(t))}(\chi)$ is $[\alpha(t)]\in \oline{\Br}(M^{(p)}/K)$, 
the specialization at $t=0$ 
corresponds to  the trivial $R$-equivalence class $\alpha(0)\sim 0$, and that at $t=1$ corresponds to the $R$-equivalence class   $\alpha(1)\sim \alpha$. This shows that the $R$-equivalence class corresponding to $\alpha$  is $R$-equivalent to the trivial element, as required. 


To this end, let $\mathcal S$ be the set of primes $\fp$ for which $M_\fp/K_\fp$ is noncyclic, fix a prime $\fq$ with $[M_\fq:K_\fq]=[M:K]$ if there exists one, and let $\alpha_\fr\in \Br(M/K)$, $\fr\in\mathcal S\setminus \{\fq\}$  (resp.\ $\fr\in \mathcal S$) be classes whose images under the map $I$ from Lemma \ref{lem:local} span the image of $I$ if $\fq$ exists (resp.\ does not exist).  Write $E=K(\eta_{s})$, where $\eta_{s}$ is the real part of a primitive $2^s$-th root of unity and consider the quadratic extension $E'=K(\eta_{s_0+1})$ of $E$. 
Since the extension $E_\fp/K_\fp$ is nontrivial for  $\fp\in \mS$, and as $E'/K$ is a cyclic $2$-power extension containing $E/K$,   $E'_\fp/E_\fp$ is a quadratic extension. Similarly, $E'_\fq/E_\fq$ is also quadratic. Therefore, the classes $\alpha_\fr$ are split locally and hence globally by $E'/K$. Thus, the class $\alpha_\fr$ is equivalent to that of a cyclic algebra $(E'/K,\sigma,a_\fr)$ for some generator $\sigma$ of $\Gal(E'/K)$ and $a_\fr\in K$. 
Since the above elements $\alpha_\fr$ generate the image of $I$, our class  $\alpha$ is equivalent in  $\oline{\Br}(M/K)$ to a class $(E'/K, \sigma,b)$ where $b$ is a product of some of the elements $a_\fr$. 

We  define  $p(t)$ 
to be a quadratic polynomial divisible by a linear polynomial $q(t)$ such that $p(0)=-1$, $p(1)=a$, $q(0)=1$ and $q(1)=b$, and let  
 $\alpha(t)$ be  the Brauer class of  $(E'(t)/K(t),\sigma,q(t))$. By definition of $q$, $\alpha(1)$ is equivalent to $\alpha$ while $\alpha(0)$ is the trivial class.  
It remains to show that $\alpha(t)$ is split by $E(t,\sqrt{-p(t)})$. Indeed, it suffices to show that the quaternion algebra $(E'(t)/E(t), q(t))$ is split by the extension $E(t,\sqrt{-p(t)})$. 
Since $q(t)\divides p(t)$, clearly $E(t,\sqrt{-p(t)})$ splits the ramification of $\alpha(t)$ for every place of $E(t)$ which is trivial on $E$, so that $\alpha(t)\otimes_{E(t)} E(t,\sqrt{-p(t)})$ is unramified at such places. 
As $E(t,\sqrt{-p(t)})$ is the function field of a curve of genus $0$, ABF implies that $\alpha(t)$ is a constant class from $\Br(E)$. Since at $t=0$, $\alpha(t)$ is trivial, this constant class is trivial, so that indeed $\alpha(t)$ is split by $E(t,\sqrt{-p(t)})$. 
\end{proof}

    We note that the proof of Theorem \ref{thm:main} may work over other fields $K$ as well, as long as every class in the quotient of $\oline{\Br}(M/K)$ lifts to a class that is split by $E'$. 

\section{Examples}
\subsection{The cyclic group of order 8}\label{sec:exam1}
We start with an example of a constant group $G=C_8$ and a number field for which $\HG^1(K,G)$ is not $R$-trivial. 
\begin{example}\label{exam:}
We claim this is the case for a quadratic number field $K=\mQ(\sqrt{a})$ for  an integer $a\equiv 1$ mod $8$, e.g.\ $a=17$. Secondly, we claim this yields examples of two Galois extensions $E_1,E_2/K$ with group isomorphic to  $C_8$ that cannot be obtained as the specialization of a single Galois extension $E/K(t)$ with group isomorphic to $C_8$. 

To prove the first claim, note that the prime $2$ splits in $K$ since $a\equiv 1$ mod $8$. Let $\fp_1,\fp_2$ be the two primes of $K$ over $2$. Thus, the field 
$M=M_{1}:=K(\sqrt{a},i)$ induces biquadratic extensions $M_{\fp_i}/K_{\fp_i}$ upon completion at $\fp_i$, $i=1,2$. 
By Proposition \ref{prop:projective}, the $R$-equivalence classes on $\HG^1(K,G)$ are in bijection with the quotient of $\Br(M/K)$ by 
$\Br(K[\sqrt{-1}]/K)+\Br(K[\sqrt{2}]/K)$. By Lemma \ref{lem:local}, this quotient is isomorphic to $C_2^{k-1}$, where $k$ is the number of places $\nu$ of $K$ for which $M_\nu/K_\nu$ is noncyclic. Since $M/K$ is unramified away from $\fp_1,\fp_2$, we have $k\leq 2$. Since $[\mQ_2(\sqrt{2},i):\mQ_2]=4$, we have $[M_{\fp_i}:K_{\fp_i}]=[M:K]=4$ and hence $\Gal(M_{\fp_i}/K_{\fp_i})\cong \Gal(M/K)$ is noncyclic, so that $k=2$. Thus, there are two $R$-equivalence classes on $\HG^1(K,G)$, as claimed.

To see the second claim, let $\varphi_1,\varphi_2\in \HG^1(K, G)$ correspond to the two distinct $R$-equivalence classes, and $E_1,E_2/K$ the two $G$-Galois extensions corresponding to  $\varphi_1,\varphi_2$, resp., see \S\ref{sec:setup}.  Since $\varphi_1,\varphi_2$ are in different $R$-equivalence classes, their images in $\HG^1(K,S)$ are distinct by Lemma \ref{lem:BG-r-equiv}. If there was a $G$-Galois extension $E/K(t)$ that specializes to both $E_1,E_2$, say at $t_1,t_2\in K$, resp., then any $\varphi\in \HG^1(K(t),G)$ that induces $E/K(t)$ specializes at $t_1,t_2$ to $\varphi_1',\varphi_2'\in \HG^1(K,G)$ 
where $\varphi_i'$ is the composition of $\varphi_i$ with an automorphism of $G$, $i=1,2$. Since an automorphism of $C_8$ sends a generator to an odd power of it, we have $\varphi_i'=\varphi_i^{j_i}$ for odd $j_i$, and hence $\varphi_i'$ and $\varphi_i$ map to the same element in $\HG^1(K,S)\cong C_2^{k-1}$, for $i=1,2$, contradicting that the images of $\varphi_1$ and $\varphi_2$ in $\HG^1(K,S)$ are distinct. Thus such  $E/K(t)$ does not exist.
\end{example}

Over $p$-adic fields, one has the following non-$R$-trivial examples:
\begin{example}\label{exam:CT}
Let $a=p=3$ and $K=\mQ_p$. We claim that $\HG^1(K,C_8^{(a)})$ is not $R$-trivial, yielding a nonconstant version of the examples desired at the end of \cite[Appendix]{CT2}. 

To see the claim note that $(3)$ is inert in $E=K(\eta_3)/K$ and  is ramified in   $K(\sqrt{-a})/K$, so that $K(\mu_8)/K$ is a noncyclic extension. Thus, the claim follows from the combination of Theorem \ref{thm:2-power} and Remark \ref{rem:p-adic}.  
\end{example}

\subsection{Semidirect products of abelian groups}\label{sec:example2}
Let $p$ be an odd prime. We next give an example a number field $K$ and a split group extension $E$ of constant abelian $p$-groups $G,A$:
$$ \xymatrix{
1\ar[r] & A \ar[r] & E\ar[r] & G\ar[r] & 1, 
}
$$
such that $\HG^1(K,E)$ is not $R$-trivial. We claim this is the case when, similarly to \cite[Ex.\ 5.4]{DLN}, one picks  $G=(C_p)^3$ and: 
\begin{itemize}
\item $A$ is the dual $\hat I=\Hom(I,\mQ/\Z)$ of the augmentation ideal $I\lhd (\mathbb Z/p^3)[G]$; and
\item $K:=K_0(\mu_{p^3})$, where $K_0/\mQ$ is a quadratic extension in which $p$ splits. 
\end{itemize}
The claim follows directly from the following proposition. Let 
$$\Sha^1_{bic}(G,\hat A):=\ker\left[\HG^1(G,\hat A)\ra\prod_{H\leq G}\HG^1(H,\hat A)\right],$$
where $H$ runs over bicyclic subgroups of $G$. Let $\exp(A)$ denote the exponent of $A$. 
\begin{prop}\label{prop:DLN}
Let $G$ be a finite abelian group and $A$ a $G$-module for which 
$\Sha^1_{bic}(G,\hat A)$ 
is nontrivial. 
Let $K$ be a number field containing  $\mu_m$ for $m:=\exp(A)$, such that $G$ appears as the Galois group of a field extension of $K_{\fp_i}$ for two  primes $\fp_i$, $i=1,2$ of $K$. Then $\HG^1(K,A\rtimes G)$ is not $R$-trivial. 
\end{prop}
Note that assumptions of the proposition hold for the above example: for, $\Sha^1_{bic}(G,\hat A)$ is nontrivial by \cite[Lem.\ 5.5]{DLN}; and $C_p^3$ appears as a Galois group over $\mQ_p(\mu_{p^3})$, and hence over $K_{\fp_1}\cong K_{\fp_2}\cong\mQ_p(\mu_{p^3})$. 

The proof requires the following version of Krasner's lemma for torsors:
\begin{lemma}\label{lem:krasner}
Let $G$ be a finite group scheme over a field $K_v$ which is complete with respect to a discrete valuation $v$, and suppose that we have a $G$-torsor $P_t$ over $K_v[t][f^{-1}]$ for some polynomial $f$. Let $t_0 \in K_v$. Then there exists $\epsilon > 0$ such that whenever $|t_1 - t_0|_v < \epsilon$, with $f(t_1), f(t_0) \neq 0$, we have that $P_{t_1}$ and $P_{t_0}$ are isomorphic as $K_v$ torsors.
\end{lemma}
\begin{proof}
Let $R = K_v[t][f^{-1}]$ and 
let $P_0 = P_{t_0} \times_{\Spec(K_v)} \Spec(R)$. 
Consider the functor $\mathcal I$ which associates to an $R$-algebra extension $S/R$ the set 
\[\mathcal I(S) = \{\phi: (P_0)_S \to (P_t)_S \mid \text{$\phi$ is an isomorphism of $G$-torsors}\}.\] 
We see that this defines a right $G$-torsor via precomposition, and hence may be written as $\mathcal I = \Spec(E)$ for some finite \'etale algebra $E/R$. Let $m$ be the rank of $E$ over $R$.

We claim that (after possibly changing $f$) we may choose a single monic polynomial $g \in R[X]$ such that $E = R[X]/g$ is of degree $m$. To see this, we set $E_{t_0} = E \otimes_R R/(t - t_0)$ and note that $R/(t - t_0) \cong K_v$ and as $E_{t_0}/K_v$ is \'etale, for a Zariski dense set of choices of $\alpha \in E_{t_0}$, we have $E_{t_0} = K_v(\alpha)$. Choose such an $\alpha \in E_{t_0}$ and let $\til \alpha \in E$ be a lift. If $U \subset \Spec R$ is the open set over which $1, \til \alpha, \til \alpha^2, \ldots, \til \alpha^{m-1}$ are a basis for $E/R$, then after possibly replacing $f$ by $f'$ for another polynomial $f'$ which doesn't vanish at $t_0$, we may assume that $U = \Spec(R)$.

We therefore can write $E = K_v[t][f^{-1}][X]/(X^m + a_{m-1}(t) X^{m-1} + \cdots + a_0(t))$ for some $a_i(t) \in K_v[t][f^{-1}]$. By definition of $\mathcal I$,
\begin{multline*}
\Hom_{K_v\text{-alg}}(E \otimes_R R/(t - t_0), K_v) = \Hom_{K_v\text{-sch}}(\Spec(K_v), \mathcal I \times_{\Spec(R)} \Spec(R/(t - t_0))) \\ = \{\phi: P_{t_0} \overset{\sim}\to P_{t_0}\} \neq \emptyset.   
\end{multline*}
Consequently, the polynomial $X^m + a_{m-1}(t_0) + \cdots + a_{0}(t_0)$ has a root, say $\alpha_0$, in $K_v$. But therefore by Krasner's Lemma, if we choose $t_1$ sufficiently close to $t_0$, we find that $X^m + a_{m-1}(t_1) + \cdots + a_{0}(t_1)$ will still have a root in $K_v$, which then tells us that 
\[\{\phi: P_{t_1} \overset{\sim}\to P_{t_0}\} = \Hom(E \otimes_R R/(t - t_1), K_v) \neq \emptyset\]
as desired.
\end{proof}

\begin{proof}[Proof of Proposition \ref{prop:DLN}]
The proof follows that of \cite[Proof of Thm.\ 5.1]{DLN}, cf.\ also \cite{Riv}. 
Let $S$ be the set of primes of $K$ lying over $p$. Note that since $\Sha^1_{bic}(G,\hat A)\neq 0$, the group $G$ is not bicyclic and hence appears as a Galois group of a field extension of $K_{v}$ only for places $v$ over $p$. 

For a nontrivial $\gamma\in \Sha^1_{bic}(G,\hat A)$, and $\fq \in S$, we  construct an invariant $d_{\gamma,\fq}(\psi)\in \mQ/\Z$ that is preserved as $\psi$ runs through specialization of an element in $\HG^1(K(t),E)$. 
Following \S\ref{sec:twists}, for $c\in \ZG^1(K,G)=\HG^1(K,G)$ (note that the equality holds since $G$ is abelian),  let ${}_{c}A$ denote the twist of $A$ by $c$, and recall it is the fiber of the projection ${}_cE\ra {}_cG$. Note that ${}_cG=G$ as $G$ is abelian.  
Setting $\Sha^1_S(K,A)=\ker\left(\HG^1(K,A)\ra\prod_{\fp\not\in S}\HG^1(K_\fp,A)\right)$, we consider the nondegenerate bilinear pairing:
\begin{equation}\label{equ:pairing} \langle\,,\rangle_c: \Sha^1_S(K,{}_{c}\hat A)/\Sha^1_\emptyset(K,{}_{c}\hat A)\times \prod_{\fp\in S}\HG^1(K_\fp,{}_cA)\ra\mQ/\Z,\end{equation}
induced by the restriction map $\Sha^1_S(K,{}_{c}\hat A)\ra  \prod_{\fp\in S}\HG^1(K_\fp,{}_c\hat A)$ (whose kernel is $\Sha^1_\emptyset(K,{}_c\hat A)$), and summation over the local Poitou--Tate pairings 
$$\langle ,\rangle_{{c(\fp )}}:\HG^1(K_\fp,{}_c\hat A)\times \HG^1(K_\fp,{}_cA)\ra\mQ/\Z,\text{ for }\fp\in S.$$  By the Poitou--Tate theorem \cite[Prop.\ 2.5]{DLN}, the right kernel of the pairing $\langle , \rangle_c$ is the image of the restriction map $\HG^1(K,{}_cA)\ra \prod_{\fp\in S}\HG^1(K_\fp,{}_cA)$. 

To construct the invariant $d_{\gamma,\fq}$, we associate to $\psi\in \HG^1(K,E)$ its image  $c\in \HG^1(K,G) = \ZG^1(K, G)$, and a class $\chi\in \HG^1(K,{}_cA)$ which maps to ${}_c\psi$. There exists such a (unique) class since the image of $\HG^1(K,{}_cA)$ coincides with the fiber of $\HG^1(K,{}_cE)\ra \HG^1(K,G)$ over $c$ as in Example \ref{exam:twist-abelian}. 
Let $c(\fq)\in\HG^1(K_{\fq},G)$ be the restriction $c$, and similarly define $\chi(\fq)$. 
The invariant $d_{\gamma,\fq}(\psi)$ is then defined to be $\langle c(\fq)^*(\gamma),\chi(\fq)\rangle_{c(\fq)}$, where 
$c(\fq)^*:\HG^1(G,{}_c\hat A)\ra\HG^1(K_\fq,{}_c\hat A)$ is the inflation map. 

We claim that $d_{\gamma,\fq}(\psi)=d_{\gamma,\fq}(\varphi)$ when $\varphi,\psi$ are the specializations of an element in $\HG^1(K(t),E)$ 
at $t\mapsto t_1$ and $t\mapsto t_2$, resp. 
As in Step III  of \cite[Thm.\ 5.1]{DLN}, the inflation $c^*(\gamma)\in \Sha^1_S(K,{}_c\hat A)$ of $\gamma$ is nontrivial, and 
the image $(\chi(\fp))_{\fp\in S}\in \prod_{\fp\in S}\HG^1(K_\fp,{}_cA)$ of $\chi$ is orthogonal to it, so that 
$\sum_{\fp\in S}\langle c(\fp)^*(\gamma),\chi(\fp)\rangle_{c(\fp)}=0.$
Since $d_{\gamma,\fq}(\psi)=\langle c(\fq)^*(\gamma),\chi(\fq)\rangle_{c(\fq)}$ by definition,  we have:
\begin{equation}\label{equ:d-equ}\sum_{\fp\in S\setminus\{\fq\}}\langle c(\fp)^*(\gamma),\chi(\fp)\rangle_{c(\fp)}=-d_{\gamma,\fq}(\psi).
\end{equation}

Now pick a specialization $t\mapsto t_0$ for $t_0\in K$ which is $\fq$-adically close to $t_1$ and $\fp$-adicaly close to $t_2$ for every $\fp\in S\setminus\{\fq\}$. Then by Lemma \ref{lem:krasner} implies that the restirctions of $\xi:=\psi(t_0)$ to $\HG^1(K_{\fq},E)$ and $\HG^{1}(K_{\fp},E)$,  $\fp\in S\setminus \{\fq\}$ coincide with $\varphi(\fq)$ and $\psi(\fp)$, resp. Thus, its image $c_\xi\in \HG^1(K,G)$ has restrictions $c(\fp)$ to $\HG^1(K_{\fp},G)$, $\fp\in S\setminus\{\fq\}$ and $c_\varphi(\fq)$ to $\HG^1(K_{\fq},G)$, 
where $c_\varphi\in \HG^1(K,G)$ is the restriction of $\varphi$. Moreover, the preimage $\chi_\xi\in \HG^1(K,{}_{c_\xi}A)$ of $\xi$ has restrictions $\chi(\fp)$ to $\HG^1(K_\fp,{}_{c(\fp)}A)$, $\fp\in S\setminus\{\fq\}$ and $\chi_\varphi(\fq)$ to $\HG^1(K_{\fq},{}_{c_\varphi(\fq)}A)$, where $\chi_\varphi$ is the preimage of $\varphi$ in $\HG^1(K,{}_{c_\varphi}A)$. Thus,  
\begin{align*}
    0=\sum_{\fp\in S}\langle c_\xi(\fp)^*(\gamma),\chi_\xi(\fp)\rangle_{c_\xi(\fp)} & = \langle c_\xi(\fq)^*(\gamma),\chi_\xi(\fq)\rangle_{c_\xi(\fq)} + \sum_{\fp\in S\setminus\{\fq\}}\langle c_\xi(\fp)^*(\gamma),\chi_\xi(\fp)\rangle_{c_\xi(\fp)} \\
     & = \langle c_\varphi(\fq)^*(\gamma),\chi_\varphi(\fq)\rangle_{c_1(\fq)} + \sum_{\fp\in S\setminus\{\fq\}}\langle c(\fp)^*(\gamma),\chi(\fp)\rangle_{c(\fp)} \\
     & = d_{\gamma,\fq}(\varphi) - d_{\gamma,\fq}(\psi), 
\end{align*} where the last equality follows from  \eqref{equ:d-equ}. Hence $d_{\gamma,\fq}(\psi)=d_{\gamma,\fq}(\varphi)$,  proving the claim.

For the trivial morphism $\mathbbm{1}\in \HG^1(K,E)$, one always has $d_{\gamma,\fq}(\mathbbm{1})=0$. We claim that there exist $\fq\in S$ and  $\psi\in \HG^1(K,E)$ such that $d_{\gamma,\fq}(\psi)\neq 0$. Since $d_{\gamma,\fq}$ is invariant, this implies that $\psi$ is not $R$-equivalent to $\mathbbm{1}$ as required. 
To prove the claim, 
fix an epimorphism $c\in \HG^1(K,G)$ whose restrictions $c(\fp_i)$ to $\HG^1(K_{\fp_i},G)$, $i=1,2$ are surjective. 
For this, we first show $\Sha^1_{\fq}(K,{}_c\hat A)=0$: 
Letting $L$ be the fixed field of $\ker c$ and $Q$ the set of primes of $L$ over $\fq$, 
note that the action of $\Gamma_L$ on $\hat A$ is trivial since $K$ contains the roots of unity of order $p^3=\exp(A)$, and hence $\Sha^1_Q(L,{}_c\hat A)=\ker\left(\Hom(\Gamma_L,\hat A)\ra\prod_{\fp\notin Q}\Hom(\Gamma_{L_\fp},\hat A)\right)$ is trivial by Chebotarev's theorem.
Thus, every $\chi\in \Sha^1(K,{}_c\hat A)$ comes from some $\oline \chi\in\HG^1(G,{}_c\hat A)$ under inflation $c^*$. 
As $c(\fp_i)\in \HG^1(K_{\fp_i},G)$, $i=1,2$ are surjections, we shall reach a contradiction using the diagram whose vertical maps are given by restriction:
$$\xymatrix{
0\ar[r] & \HG^1(G,{}_c\hat A)\ar[r]^{c^*} \ar[d]_{\cong} & \HG^1(K,{}_c\hat A) \ar[d] \\
0\ar[r] &  \HG^1(G_{\fp_i},{}_{c(\fp_i)}\hat A)\ar[r]^{c(\fp_i)^*} & \HG^1(K_{\fp_i},{}_{c(\fp_i)}\hat A). 
}
$$ Indeed, the surjectivity of $c(\fp_i)$ implies the restriction $\HG^1(G,{}_c\hat A)\ra \HG^1(G_{\fp_i},{}_{c(\fp_i)}\hat A)$ is an isomorphism and hence the image of $\oline \chi$ is nontrivial and hence is mapped to a nontrivial element under $c(\fp_i)^*$. This  contradicts the triviality of the image of $\chi\in \Sha^1(K,{}_c\hat A)$ in   $\HG^1(K_{\fp_i},{}_{c(\fp_i)}\hat A)$ whenever $\fp_i\neq \fq$. 
Thus, $\Sha^1_\fq(K,{}_c\hat A)=0$.

Since the Poitou--Tate pairing \eqref{equ:pairing} is nondegenerate, the above triviality of $\Sha^1_\fq(K,{}_c\hat A)$ implies $\res_\fq:\HG^1(K,{}_cA)\ra\HG^1(K_\fq,{}_cA)$ is surjective, for every $\fq\in S$. 
Since $c^*(\gamma)\in \Sha^1_S(K,{}_c\hat A)$ is nontrivial as above, and the local Poitou--Tate pairing is nondegenerate, there exist a prime  $\fq\in S$ and $\chi(\fq)\in \HG^1(K_\fq,{}_cA)$ such that  $\langle c(\fq)^*(\gamma),\chi(\fq)\rangle_{c(\fq)}\neq 0$. As $\res_\fq$ is surjective, 
we may pick $\chi\in \HG^1(K,{}_cA)$ that restricts to $\chi(\fq)\in \HG^1(K_\fq,{}_cA)$. Letting $\psi$ be the image of $\chi$ under $\HG^1(K,{}_cA)\ra\HG^1(K,{}_cE)$, we get that $d_{\gamma,\fq}(\psi)\neq 0$, proving the claim. 
\end{proof}
The proof of the above theorem suggests a deeper relation between $R$-equivalence and weak approximation that remains to be investigated.

\end{document}